\documentclass[11pt, leqno]{article}

\usepackage[sort]{cite}       
\usepackage{amssymb}          
\usepackage{latexsym}         
\usepackage{amsmath}          
\usepackage{amscd}            
\usepackage[latin1]{inputenc} 
\usepackage{eucal}            
\usepackage{bbm}              
\usepackage{theorem}          
\usepackage{stmaryrd}         
\usepackage{mathrsfs}         
\usepackage{paralist}         
\usepackage[all,dvips]{xy}

\title{Deformation Quantization of Surjective Submersions\\ and Principal Fibre Bundles}

\author{\textbf{Martin Bordemann}\thanks{Martin.Bordemann@uha.fr},
  \\[0.1cm]
  Laboratoire de Math{\'e}matiques\\
  Universit{\'e} de Haute-Alsace Mulhouse \\
  4, Rue des Fr{\`e}res Lumi{\`e}re \\
  F.68093 Mulhouse \\
  France\\[0.2cm]
  \textbf{and}
  \\[0.2cm]
  \addtocounter{footnote}{1}
  \textbf{Nikolai Neumaier}\thanks{Nikolai.Neumaier@physik.uni-freiburg.de},
  \textbf{Stefan Waldmann}\thanks{Stefan.Waldmann@physik.uni-freiburg.de},
  \textbf{Stefan Weiss}\thanks{Stefan.Weiss@physik.uni-freiburg.de}
  \\[0.1cm]
  Fakult{\"a}t f{\"u}r Mathematik und Physik\\
  Albert-Ludwigs-Universit{\"a}t Freiburg\\
  Physikalisches Institut\\
  Hermann Herder Stra{\ss}e 3\\
  D 79104 Freiburg\\
  Germany
  }

\date{November 2007\\[1cm]
  To the memory of Julius Wess
}

\renewcommand{\mathbb}[1]{\mathbbm{#1}} 

\numberwithin{equation}{section}

\newcounter{comment}

\theoremheaderfont{\normalfont\bfseries}
\theorembodyfont{\itshape}
\newtheorem{lemma}{Lemma}[section]
\newtheorem{proposition}[lemma]{Proposition}
\newtheorem{theorem}[lemma]{Theorem}
\newtheorem{corollary}[lemma]{Corollary}

\theorembodyfont{\rmfamily}

\newtheorem{definition}[lemma]{Definition}
\newtheorem{example}[lemma]{Example}
\newtheorem{remark}[lemma]{Remark}

\newenvironment{proof}[1][{}]{\par\noindent\textsc{Proof{#1}: }}{\hspace*{\fill}$\blacksquare$\smallskip\noindent\par}

\newcommand{\id}         {\operatorname{\mathsf{id}}}

\newcommand{\Lie}        {\operatorname{\mathscr{L}\!}}    
\newcommand{\supp}       {\operatorname{\mathrm{supp}}}

\newcommand{\Hom}        {\operatorname{\mathsf{Hom}}}   
\newcommand{\End}        {\operatorname{\mathsf{End}}}

\newcommand{\I}          {\mathrm{i}}
\newcommand{\E}          {\mathrm{e}}
\newcommand{\D}          {\operatorname{\mathrm{d}}}
\newcommand{\lie}[1]     {\mathfrak{#1}}

\newcommand{\Anti}       {\Lambda}
\newcommand{\ins}        {\operatorname{\mathrm{i}}}

\newcommand{\HdR}        {\mathrm{H}_{\scriptscriptstyle\mathrm{dR}}}

\newcommand{\starp}     {\mathbin{\star'}}
\newcommand{\bulletp}   {\mathbin{\bullet'}}

\newcommand{\Def}      {\mathrm{Def}}

\newcommand{\HC}       {\mathrm{HC}}
\newcommand{\HCtype}   {\mathrm{HC}_{\scriptstyle{\mathrm{type}}}}

\newcommand{\HCdiff}   {\mathrm{HC}_{\scriptstyle{\mathrm{diff}}}}

\newcommand{\HHtype}   {\mathrm{HH}_{\scriptstyle{\mathrm{type}}}}

\newcommand{\Diffop}   {\operatorname{\mathrm{DiffOp}}}
\newcommand{\Diffopver}{\operatorname{\mathrm{DiffOp}_{\mathrm{ver}}}}
\newcommand{\Union}    {\operatorname*{\mbox{$\bigcup$}}}

\newcommand{\bullett}   {\mathbin{\tilde{\bullet}}}

\begin{document}

\maketitle

\begin{abstract}
    In this paper we establish a notion of deformation quantization of
    a surjective submersion which is specialized further to the case of
    a principal fibre bundle: the functions on the total space are
    deformed into a right module for the star product algebra of the
    functions on the base manifold. In case of a principal fibre
    bundle we require in addition invariance under the principal
    action. We prove existence and uniqueness of such deformations.
    The commutant within all differential operators on the total space
    is computed and gives a deformation of the algebra of vertical
    differential operators. Applications to noncommutative gauge field
    theories and phase space reduction of star products are discussed.
\end{abstract}

\newpage

\tableofcontents

%
%

\section{Introduction}
\label{sec:Intro}

Principal fibre bundles and surjective submersions are omnipresent in
differential geometry. To name just a few instances: any vector bundle
is an associated bundle to its frame bundle, proper and free Lie group
actions are principal fibre bundles, and physical gauge theories are
formulated using principal fibre bundles as starting point. The
projection map from the total space of the principal bundle to the
base space is an example of a surjective submersion. Among many
others, an interesting example of a surjective submersion in phase
space reduction in geometrical mechanics is the projection map from
the momentum level surface in a Poisson manifold onto the reduced
phase space whenever one is in the regular situation.

On the other hand, deformation quantization \cite{bayen.et.al:1978a}
has reached great popularity in various applications in mathematical
physics, not only involving the original intention of understanding
the quantum theory of a classical mechanical system with phase space
modelled by a Poisson manifold. Even though this original motivation
is still one of the major issues in deformation quantization, more
recent applications involve the so-called noncommutative space-times,
see e.g.\ \cite{doplicher.fredenhagen.roberts:1995a} .  Here the
space-time manifold is endowed with a noncommutative deformation, the
star product, which is seen as an effective theory of some still
unknown quantum gravity. Then on such a noncommutative space-time
(quantum) field theories are established and studied intensely. Of
particular interest are of course again the gauge theories.

Having these applications in mind it is natural to ask how one can
define and construct deformation quantizations of principal bundles
and, more generally, of surjective submersions.

In the following we will exclusively work with \emph{formal}
deformation quantization: recall that a formal star product $\star$ on
a manifold $M$ is an associative $\mathbb{C}[[\lambda]]$-bilinear
product for the formal power series $C^\infty(M)[[\lambda]]$ in
$\lambda$ with values in the smooth complex-valued functions
$C^\infty(M)$ such that
\begin{equation}
    \label{eq:StarProduct}
    f \star g = \sum_{r=0}^\infty \lambda^r C_r(f, g),
\end{equation}
$C_0(f, g) = fg$ is the pointwise product and $1 \star f = f = f \star
1$. We only consider differential star products where $C_r$ is a
bidifferential operator for all $r$.  It follows from associativity
that $\{f, g\} = \frac{1}{\I} (C_1(f, g) - C_1(g, f))$ is a Poisson
bracket on $M$. Conversely, any Poisson bracket can be deformed
(quantized) into a star product, this is a consequence of Kontsevich's
famous formality theorem \cite{kontsevich:1997:pre, kontsevich:2003a}.
For an elementary introduction to Poisson geometry and deformation
quantization see e.g.\  \cite{waldmann:2007a}.

In order to find a reasonable definition of a deformation quantization
of a principal fibre bundle one can rely on several other approaches,
some of which we shall recall now:

In many approaches to gauge theories on noncommutative space-times one
can read between the lines and possibly extract a good definition of a
deformation quantization of a principal fibre bundle. However, it is
not completely obvious as here either only local aspects are
discussed, typically the Weyl-Moyal star product on a symplectic
vector space, or only particular structure groups like $\mathrm{Gl}(n,
\mathbb{C})$ or $\mathrm{U}(n)$. Here in particular the works of
Jur\v{c}o, Schupp, Wess and co-workers, see e.g.
\cite{jurco.schupp.wess:2002a, jurco.schupp.wess:2001a,
  jurco.schraml.schupp.wess:2000a, jurco.schupp.wess:2000a,
  jurco.schupp:2000a}, are discussed and developed further in the
physics literature. They seem to give promising models for gauge
theories on noncommutative space-times, see also
\cite{waldmann:2007b:pre} for a review on the geometric nature of such
noncommutative field theories.

In \cite{bursztyn.waldmann:2000b, bursztyn:2001a, bursztyn:2002a,
  waldmann:2002a, waldmann:2001b} the deformation quantization of
vector bundles was discussed in detail, where indeed deformed
transition functions were found: this indicates a deformation
quantization of the corresponding frame bundle. However, a global
description of the deformed frame bundle is still missing and
structure groups beyond the general linear group do not seem to be
accessible by this approach. Also in \cite{hawkins:2000a,
  hawkins:1999a} the deformation theory of vector bundles in the
context of strict deformation quantizations is discussed.

Furthermore, as quantum analogue of principal bundles the so-called
\emph{Hopf-Galois extensions} are studied in detail: here the total
space $P$ is replaced by some algebra $\mathcal{P}$, the structure Lie
group $G$ is replaced by a Hopf algebra $H$ acting (or co-acting) on
$\mathcal{P}$ and the base space $M$ is replaced by the sub-algebra
$\mathcal{M}$ of $H$-invariant elements in $\mathcal{P}$ together with
some additional, more technical properties encoding the freeness and
properness of the action, see e.g.\ \cite{brzezinski.majid:1993a} and
\cite{hajac.matthes.szymanski:2003a, dabrowski.grosse.hajac:2001a} for
recent developments. Even though this is a very successful approach
mainly used in a $C^*$-algebraic formulation, there are simple
examples in deformation quantization where this does not seem
appropriate. Taking the idea of Hopf-Galois extensions literally would
lead to the following definition of a deformation quantization of a
principal fibre bundle $\mathsf{p}: P \longrightarrow M$. Given a star
product $\star$ on $M$ one should try to find a star product $\star_P$
on $P$ such that the pull-back $\mathsf{p}^*$ can be deformed into an
algebra homomorphism. Then of course, the Hopf algebra deformation and
the further technical requirements still have to be found and
satisfied. However, already for the first step one finds hard
obstructions:
\begin{example}
    \label{example:NoHopfGalois}
    Let $\mathsf{p}: P \longrightarrow M$ be a surjective submersion
    and $\star$ a star product on $M$ quantizing a Poisson structure
    $\pi_M \in \Gamma^\infty(\Anti^2 TM)$. Assume that there exists a
    star product $\star_P$ on $P$ such that $\mathsf{p}^*$ allows for
    a deformation into an algebra homomorphism
    \begin{equation}
        \label{eq:Deformpobenstern}
        \mathsf{p}^* + \cdots: 
        (C^\infty(M)[[\lambda]], \star) \longrightarrow
        (C^\infty(P)[[\lambda]], \star_P).
    \end{equation}
    Let $\pi_P$ denote the Poisson structure on $P$ determined by
    $\star_P$. Then by a simple computation, $\mathsf{p}: (P, \pi_P)
    \longrightarrow (M, \pi_M)$ is a Poisson map.
\end{example}
Since in this situation already the lowest orders have to satisfy a
condition, one has to expect obstructions. Indeed, this happens
already in the following simple example:
\begin{example}
    \label{example:HopfFibration}
    Consider the Hopf fibration $\mathsf{p}: S^3 \longrightarrow S^2$,
    which is a $S^1$-principal bundle, and a \emph{symplectic} Poisson
    structure $\pi_{S^2}$ on the $2$-sphere $S^2$. Then there is no
    Poisson structure on $S^3$ making $\mathsf{p}$ a Poisson map.
    Indeed, it is easy to see that the symplectic leaves of $(S^3,
    \pi_{S^3})$ have to be $2$-dimensional and the restriction of
    $\mathsf{p}$ to one leaf is still surjective. Thus the leaf covers
    $S^2$ whence it is diffeomorphic to $S^2$ via $\mathsf{p}$.  But
    this immediately gives a global section of the non-trivial
    principal bundle $\mathsf{p}: S^3 \longrightarrow S^2$, a
    contradiction. Of course this relies very much on the fact that we
    have chosen a symplectic Poisson structure on $S^2$. There are
    examples of Hopf-Galois extensions deforming the Hopf fibration
    where (necessarily) the Poisson structure on $S^2$ is not
    symplectic.
\end{example}
We will come back to this example in
Section~\ref{sec:AssociatedBundles} and study the above obstruction
from a more sophisticated point of view: We will see that also a
deformation of $\mathsf{p}^*$ into a \emph{bimodule} structure will
result in hard obstructions.

Before we give the final definition, we recall the third motivation
coming from deformation quantization itself: consider a big phase
space $\tilde{M}$ with a coisotropic submanifold $\iota: P
\longrightarrow \tilde{M}$. If the characteristic foliation of $P$ is
well-behaved enough, then the leaf space $M = P \big/ \sim$ itself is
a manifold such that the projection $\mathsf{p}: P \longrightarrow M$
is a surjective submersion. It is a well-known fact that $M$ inherits
a Poisson structure from $\tilde{M}$. A particular case is obtained if
$\tilde{M}$ is equipped with a Hamiltonian group action of some Lie
group $G$ with equivariant momentum map $J: \tilde{M} \longrightarrow
\lie{g}^*$ and $P = J^{-1}(0)$ is the momentum level zero surface. In
this case $G$ acts also on $P$ and $M$ is the quotient $P \big/ G$. In
the free and proper case we have a good reduced phase space $M$ and
$P$ is a $G$-principal bundle over $M$. This situation is the famous
Marsden-Weinstein reduction, see
e.g.\  \cite{abraham.marsden:1985a}, Sect.~4.3.

When it comes to deformation quantization of this picture, one wants
to find a star product $\tilde{\star}$ on $\tilde{M}$ which is
compatible with the `constraint surface' $P$ in such a way, that one
can construct a star product $\star$ on $M$ out of $\tilde{\star}$.
Several options for this have been discussed like the BRST formalism
\cite{bordemann.herbig.waldmann:2000a ,bordemann:2000a, herbig:2006a}
or more ad hoc constructions as in \cite{gloessner:1998a:pre,
  bordemann.brischle.emmrich.waldmann:1996a,
  bordemann.brischle.emmrich.waldmann:1996b}. More recently, it became
clear that a deformation of the functions on $P$ as a \emph{bimodule}
for the deformed algebras of functions on $\tilde{M}$ and $M$,
respectively, leads to a very satisfying picture: one should try to
define a left module structure on $C^\infty(P)[[\lambda]]$ with
respect to $\tilde{\star}$ deforming the canonical one coming from
$\iota^*$ in such a way that the module endomorphisms are isomorphic
to $C^\infty(M)[[\lambda]]$ thereby inducing a star product $\star$ on
$M$. This point of view has been promoted in \cite{bordemann:2005a,
  bordemann:2004a:pre} where first results for the general symplectic
case have been obtained. For the more general Poisson case, see
\cite{cattaneo.felder:2005a:pre, cattaneo.felder:2004a,
  cattaneo:2004a}, where sufficient conditions for a successful
reduction where formulated.

Taking this last motivation into account it is clear that a
deformation of $C^\infty(P)$ into an algebra does not seem
to be appropriate at all: geometrically a deformed product would
result in a Poisson structure in first order, but in phase space
reduction there is \emph{no} Poisson structure on the constraint
surface $P$, only on $\tilde{M}$ and on $M$. In fact, on $P$ one
always has a Dirac structure, see e.g.\ the discussion in
\cite{courant:1990a, cattaneo.zambon:2007a}. Thus a deformation into a
right module seems to be more adapted to the reduction picture. In
fact, this will be our final definition:
\begin{definition}[Deformation quantization of surjective submersions]
    \label{definition:def-quant-surj-sub}
    Let $\mathsf{p}: P \longrightarrow M$ be a surjective submersion
    and $\star$ be a star product on $M$.
    \begin{enumerate}
    \item A deformation quantization of the surjective submersion is a
        $(C^\infty(M)[[\lambda]],\star)$-right module
        structure $\bullet$ of $C^\infty(P)[[\lambda]]$,
        such that
        \begin{equation}
            \label{eq:deformation-surjective-submersion}
            f \bullet a 
            = f \cdot \mathsf{p}^*a 
            + \sum_{r=1}^\infty \lambda^r \rho_r(f,a)
        \end{equation}
        for $f\in C^\infty(P)[[\lambda]]$ and $a \in
        C^\infty(M)[[\lambda]]$ with bidifferential
        operators $\rho_r$.
    \item Two such deformations $\bullet$ and $\tilde{\bullet}$ are
        called equivalent if and only if there exists a formal series
        $T = \id_{C^\infty(P)} + \sum_{r=1}^\infty \lambda^r T_r$ of
        differential operators $T_r \in \Diffop(C^\infty(P))$ such
        that for all $f \in C^\infty(P)[[\lambda]]$ and $a \in
        C^\infty(M)[[\lambda]]$
        \begin{equation}
            \label{eq:equivalence-surj-sub}
            T(f \bullet a)=
            T(f) \mathbin{\tilde{\bullet}} a.
        \end{equation}
    \item A deformation quantization $\bullet$ is said to preserve the
        fibration if
        \begin{equation}
            \label{eq:preserve-fibration-surj-sub}
            (\mathsf{p}^*a) \bullet b
            = \mathsf{p}^* (a \star b).
        \end{equation}
    \end{enumerate}
\end{definition}
Preserving the fibration will be a technical but yet convenient
condition to impose. Clearly, it is equivalent to the condition $1
\bullet a = \mathsf{p}^*a$ by the right module property. As usual in
deformation quantization we require the maps $\rho_r$ to be
bidifferential operators as already $\star$ is always assumed to be a
differential star product. One can show easily that for any right
module structure one necessarily has $f \bullet 1 = f$ since $1 \in
C^\infty(M)$ is still the unit element with respect to $\star$.  From
this definition it is easy to motivate the more specific situation of
a principal fibre bundle:
\begin{definition}[Deformation quantization of principal fibre bundles]
    \label{definition:def-quant-principal}
    Let $\mathsf{p}:P\longrightarrow M$ be a principal fibre bundle
    with structure group $G$ and principal right action $\mathsf{r}: P
    \times G \longrightarrow P$, $\mathsf{r}_g(u) = \mathsf{r}(u,g)$,
    and $\star$ a star product on $M$.
    \begin{enumerate}
    \item A deformation quantization of the principal fibre bundle is
        a $G$-invariant deformation quantization of the surjective
        submersion $\mathsf{p}: P \longrightarrow M$ with respect to
        $\star$, i.e.\  a right module structure $\bullet$ as in
        \eqref{eq:deformation-surjective-submersion} with the additional
        property
        \begin{equation}
            \label{eq:deformation-right-module-principal}
            \mathsf{r}_g^* (f \bullet a)
            = (\mathsf{r}_g^*f) \bullet a
        \end{equation}
        for all $f \in C^\infty(P)[[\lambda]]$, $a \in
        C^\infty(M)[[\lambda]]$, and $g \in G$.
    \item Two such deformations $\bullet$ and $\tilde{\bullet}$ are
        called equivalent if they are equivalent in the sense of
        definition \ref{definition:def-quant-surj-sub} with
        $G$-invariant operators $T_r$, i.e.\  in addition to
        \eqref{eq:deformation-right-module-principal} one has for all
        $g \in G$
        \begin{equation}
            \label{equivalence-principal-G-invariance}
            \mathsf{r}_g^* \circ T_r
            = T_r \circ \mathsf{r}_g^*.
        \end{equation}
    \end{enumerate}
\end{definition}

The main goal of the present paper is to prove the following two
theorems about existence and uniqueness of differential deformations
of surjective submersions and principal fibre bundles:
\begin{theorem}
    \label{theorem:def-quant-surj-sub}
    Every surjective submersion $\mathsf{p}: P \longrightarrow M$ with a
    star product $\star$ on $M$ admits a deformation quantization
    which is unique up to equivalence. Moreover, one can achieve a
    deformation which respects the fibration.
\end{theorem}
\begin{theorem}
    \label{theorem:def-quant-principal}
    Every principal fibre bundle $\mathsf{p}: P \longrightarrow M$
    with a star product $\star$ on $M$ admits a deformation
    quantization which is unique up to equivalence. Again, one can
    achieve a deformation which respects the fibration.
\end{theorem}
The proof of both theorems relies on an order by order construction of
the deformed module structures which is possible since we are able to
show that the relevant Hochschild cohomologies are trivial. With this
(non-trivial) result on the Hochschild cohomology the remaining proof
is very simple. To show the vanishing of the Hochschild cohomologies
we heavily use techniques developed in
\cite{bordemann.et.al:2005a:pre}. Alternatively, the existence of such
deformations follows also from a Fedosov-like construction as
discussed in detail in \cite{weiss:2006a} for the case where the star
product on $M$ quantizes a \emph{symplectic} Poisson bracket.

The very satisfactory answer to the existence and classification
questions indicates that our definitions for deformation quantization
of surjective submersions and principal fibre bundles are
reasonable. In a next step it remains to answer whether the
definitions are \emph{useful} once we have shown
Theorem~\ref{theorem:def-quant-surj-sub} and
\ref{theorem:def-quant-principal}. Here we have to go back to the
original motivations why a deformation quantization is desirable:
\begin{compactitem}
\item Concerning applications in gauge theories on noncommutative
    space-times we would like to see how one can formulate a global
    and geometric approach to such gauge theories, including in
    particular the notions of associated bundles, connections and
    Yang-Mills actions. Here we have partial answers where we can show
    how the process of associating vector bundles is formulated very
    naturally in our framework. The result will be a deformed vector
    bundle in the sense of \cite{bursztyn.waldmann:2000b} which
    provides the geometric formulation of matter fields in
    noncommutative field theory \cite{waldmann:2001b}. Moreover, the
    action of the infinitesimal gauge transformations can be clarified
    and compared with the approaches in
    \cite{jurco.schraml.schupp.wess:2000a}. To this end, in
    Theorem~\ref{theorem:CooleDeformation}, we compute the
    \emph{commutants} of the right modules obtained by
    Theorem~\ref{theorem:def-quant-surj-sub} and
    \ref{theorem:def-quant-principal} within all differential
    operators which turn out to be deformations of the vertical
    differential operators on the total space.  On the other hand, the
    role of connections still has to be clarified in our geometric and
    global approach.
\item Concerning the relation to the Hopf-Galois extensions we can use
    the results on the commutant to formulate more refined
    obstructions for the existence of a $G$-invariant bimodule
    deformation using results from the Morita theory of star products
    in Corollary~\ref{corollary:NoBimodule}.  Of course, in our
    situation the structure group itself is always the undeformed Lie
    group $G$ and not a Hopf algebra deformation. Clearly, further
    investigations will be necessary to understand the relations
    between these two approaches better.
\item The applications to phase space reduction of star products
    consist in finding hard obstructions: Since the right module
    deformation is unique up to equivalence the commutant is uniquely
    determined, too. In order to carry through the reduction process
    one needs to find a left module structure for the star product
    $\tilde{\star}$ of the big phase space on the functions on $P$,
    i.e.\  an algebra homomorphism into the commutant. Now, for an
    arbitrary choice of the star product $\star$ on $M$ such an
    algebra homomorphism may or may not exist, which gives a necessary
    and also sufficient condition for the reduction. Clearly, this is
    still rather inexplicit and has to be investigated in more detail.
    In particular, we plan to give a comparison with the results in
    \cite{bordemann:2004a:pre, cattaneo.felder:2005a:pre,
      cattaneo.felder:2004a, cattaneo:2004a}.
\end{compactitem}

The paper is organized as follows: In
Section~\ref{sec:AlgebraicPreliminaries} we collect some well-known
facts on algebraic deformation theory in the spirit of Gerstenhaber
and formulate the deformation problem of modules to introduce the
relevant Hochschild cohomologies. Some particular attention is put on
the fact that in the end we need more particular cochains,
bidifferential ones in our case. Section~\ref{sec:homological} recalls
some basic constructions from \cite{bordemann.et.al:2005a:pre} which
are needed to compute the Hochschild cohomologies in the local models.
Here the Koszul and the bar resolutions are recalled and some explicit
homotopies are given. Section~\ref{sec:SurjectiveSubmersions} contains
the proof of Theorem~\ref{theorem:def-quant-surj-sub} using an order
by order construction. We also compute the commutant as a deformation
of the vertical differential operators. In
Section~\ref{sec:PrincipalBundles} we prove
Theorem~\ref{theorem:def-quant-principal} including also a computation
of the commutant and the compatibility of the resulting bimodule
structure with the $G$-action. Finally, in
Section~\ref{sec:AssociatedBundles} we show how a simple tensor
product construction gives the deformation quantization of associated
vector bundles out of our deformation quantization of a principal
fibre bundle. The commutant of the deformed right module structure on
the principal fibre bundle maps onto the commutant of the deformation
quantization of the associated vector bundle.

\medskip
\noindent
\textbf{Acknowledgements:} We would like to thank Henrique Bursztyn,
Alberto Cattaneo, Simone Gutt, Brano Jur\v{c}o, Rainer Matthes, Peter
Schupp, Jim Stasheff, Julius Wess, and Marco Zambon for valuable
discussions and remarks.

%
%

\section{Algebraic Preliminaries}
\label{sec:AlgebraicPreliminaries}

In this section we recall some basic facts on the deformation theory
of algebras and modules. Essentially, all stated definitions and
results are well-known from the very first works of Gerstenhaber
\cite{gerstenhaber:1964a} and e.g.\ \cite{donald.flanigan:1974a}.
However, we will need some explicit expressions for the relevant
cochains whence we present the material in a self-contained way.

Let $\mathbb{K}$ be a field of characteristic $0$ and
$(\mathcal{A},\mu_0)$ an associative $\mathbb{K}$-algebra.
Furthermore, let $\mathcal{E}$ be a vector space over $\mathbb{K}$
with an $\mathcal{A}$-right module structure
\begin{equation}
    \label{eq:undef-right-module-structure}
    \rho_0: \mathcal{E} \times \mathcal{A} \longrightarrow 
    \mathcal{E}.
\end{equation}
We shall be interested in a formal associative deformation of the
algebra multiplication
\begin{equation}
    \label{eq:def-algebra-structure}
    \mu = \sum_{r=0}^\infty \lambda^r \mu_r:
    \mathcal{A}[[\lambda]] \times \mathcal{A}[[\lambda]] 
    \longrightarrow \mathcal{A}[[\lambda]]  
\end{equation}
in the sense of Gerstenhaber \cite{gerstenhaber:1964a} which we assume
to be given. For this given deformation $\mu$ we are looking for a
deformation of the module structure
\eqref{eq:undef-right-module-structure}, again in the framework of
formal series
\begin{equation}
    \label{eq:def-right-module-structure}
    \rho = \sum_{r=0}^\infty \lambda^r \rho_r:
    \mathcal{E}[[\lambda]] \times \mathcal{A}[[\lambda]]
    \longrightarrow \mathcal{E}[[\lambda]]
\end{equation}
such that $\rho$ is a right module structure with respect to $\mu$.
All $\mathbb{K}$-multilinear maps will be extended to
$\mathbb{K}[[\lambda]]$-multilinear maps in the following.

In this purely algebraic framework one can now derive expressions for
the obstructions to construct such deformations order by order in the
deformation parameter analogously to \cite{gerstenhaber:1964a}.
However, we shall need a slightly more specific framework: typically,
the maps $\mu_r$ have additional properties and also the $\rho_r$ are
required to have additional properties like e.g.\ continuity with
respect to some given topology. In order to formalize this we consider
Hochschild cochains of the algebra $\mathcal{A}$ of particular types.
These `types' should satisfy the following conditions which simply
allow to reproduce Gerstenhaber's arguments and computations.
\begin{enumerate}
\item We consider Hochschild cochains of a certain type which we
    denote by $\HCtype^\bullet (\mathcal{A}, \mathcal{A}) \subseteq
    \HC^\bullet(\mathcal{A}, \mathcal{A})$, where
    $\HCtype^\bullet(\mathcal{A}, \mathcal{A})$ is required to be
    closed under the usual insertions $\circ_i$ after the $i$-th
    position and $\HCtype^0(\mathcal{A}, \mathcal{A}) = \mathcal{A}$.
    Moreover, we require $\mu_0 \in \HCtype^2(\mathcal{A},
    \mathcal{A})$.
\end{enumerate}
In particular, it follows that $\HCtype^\bullet(\mathcal{A},
\mathcal{A})$ is a subcomplex of $\HC^\bullet(\mathcal{A},
\mathcal{A})$ with respect to the Hochschild differential $\delta$
corresponding to $\mu_0$. Moreover, $\HCtype^\bullet(\mathcal{A},
\mathcal{A})$ is closed under the $\cup$-product and the Gerstenhaber
bracket $[\cdot, \cdot]$ and the corresponding Hochschild cohomology
$\HHtype^\bullet(\mathcal{A}, \mathcal{A})$ is a Gerstenhaber algebra
itself. We assume in the following that the given deformation $\mu$ of
$\mu_0$ consists of cochains $\mu_r \in \HCtype^2(\mathcal{A},
\mathcal{A})$.

For the deformation problem of the module structure we consider
particular Hochschild cochains in $\HC^\bullet(\mathcal{A},
\End(\mathcal{E}))$, where $\End(\mathcal{E})$ is given the canonical
$(\mathcal{A}, \mathcal{A})$-bimodule structure. In detail, we require
the following:
\begin{enumerate}
    \addtocounter{enumi}{1}
\item $\mathcal{D} \subseteq \End(\mathcal{E})$ is a subalgebra with
    $\id \in \mathcal{D}$.
\item We consider cochains $\HCtype^\bullet(\mathcal{A}, \mathcal{D})$
    with values in the subalgebra $\mathcal{D}$ where `type' has the
    property that for $\phi \in \HCtype^k (\mathcal{A}, \mathcal{D})$
    and $\psi \in \HCtype^l (\mathcal{A}, \mathcal{A})$ we have $\phi
    \circ_i \psi \in \HCtype^{k+l-1} (\mathcal{A}, \mathcal{D})$, where
    as usual for $a_1, \ldots, a_{k+l-1} \in \mathcal{A}$
    \begin{equation}
        \label{eq:phicircipsi}
        (\phi \circ_i \psi)(a_1, \ldots, a_{k+l-1})
        = \phi(a_1, \ldots, a_i, \psi(a_{i+1}, \ldots, a_{i+l}),
        a_{i+l+1}, \ldots, a_{k+l-1}).
    \end{equation}
    Of course we want $\mathcal{D} = \HCtype^0(\mathcal{A},
    \mathcal{D})$.
\item Finally, for $\phi_1 \in \HCtype^k (\mathcal{A}, \mathcal{D})$ and
    $\phi_2 \in \HCtype^l (\mathcal{A}, \mathcal{D})$ we require $\phi_1
    \circ \phi_2 \in \HCtype^{k+l} (\mathcal{A}, \mathcal{D})$ where
    \begin{equation}
        \label{eq:phicircpsi}
        (\phi_1 \circ \phi_2)(a_1, \ldots, a_{k+l})
        =
        \phi_1(a_1, \ldots, a_k) \circ \phi_2(a_{k+1}, \ldots, a_{k+l}).
    \end{equation}
\end{enumerate}
Of course we now require that the undeformed right module structure
$\rho_0$ is a cochain $\rho_0 \in \HCtype^1(\mathcal{A},
\mathcal{D})$. Being a right module structure implies that
$\mathcal{D}$ is a $(\mathcal{A}, \mathcal{A})$-bimodule via
\begin{equation}
    \label{eq:DBimodule}
    a \cdot D \cdot b = \rho_0(b) \circ D \circ \rho_0(a),
\end{equation}
where $a, b \in \mathcal{A}$ and $D \in \mathcal{D}$. This is the
restriction of the canonical bimodule structure of
$\End(\mathcal{E})$. Thus $\HCtype^\bullet(\mathcal{A}, \mathcal{D})$
is a subcomplex of $\HC^\bullet(\mathcal{A}, \End(\mathcal{E}))$
whence we obtain a corresponding Hochschild cohomology denoted by
$\HHtype^\bullet(\mathcal{A}, \mathcal{D})$.

Within this refined framework we want to find a deformation $\rho$ as
in \eqref{eq:def-right-module-structure}, where now all $\rho_r \in
\HCtype^1(\mathcal{A}, \mathcal{D})$. Completely analogously to the
general case one obtains the following lemma:
\begin{lemma}
    \label{lemma:DeformationObstruction}
    Assume that $\rho^{(r)} = \rho_0 + \cdots + \lambda^r \rho_r$ is a
    right module structure with respect to $\mu$ up to order
    $\lambda^r$ with $\rho_s \in \HCtype^1 (\mathcal{A}, \mathcal{D})$
    for all $s = 0, \ldots, r$. Then the condition for $\rho_{r+1} \in
    \HCtype^1 (\mathcal{A}, \mathcal{D})$ to define a right module
    structure $\rho^{(r+1)} = \rho^{(r)} + \lambda^{r+1} \rho_{r+1}$
    up to order $\lambda^{r+1}$ is
    \begin{equation}
        \label{eq:ConditionRhoOrderByOrder}
        \delta \rho_{r+1} = R_r
    \end{equation}
    with $R_r \in \HCtype^2 (\mathcal{A}, \mathcal{D})$ explicitly
    given by
    \begin{equation}
        \label{eq:ObstructionExistence}
        R_r (a, b) 
        =  
        \sum_{s=0}^r \rho_s (\mu_{r+1-s}(a, b))
        - \sum_{s=1}^r \rho_s(b) \circ \rho_{r+1-s}(a).
    \end{equation}
    Moreover, $\delta R_r = 0$ whence the obstruction in order
    $\lambda^{r+1}$ is the class $[R_r] \in \HHtype^2(\mathcal{A},
    \mathcal{D})$.
\end{lemma}
\begin{proof}
    The only new aspect is that $R_r \in \HCtype^2(\mathcal{A},
    \mathcal{D})$ which is clear from the explicit formula and the
    conditions \textit{ii.)}-\textit{iv.)}.
\end{proof}
In particular, if $\HHtype^2(\mathcal{A}, \mathcal{D}) = \{0\}$, an
order by order construction immediately yields the existence of a
deformation $\rho$ of the desired type.

In a next step, we consider two deformed right module structures
$\rho$ and $\tilde{\rho}$ of the given type for the same associative
deformation $\mu$ of $\mathcal{A}$. Then they are called
(cohomologically) equivalent if there exists a formal series
\begin{equation}
    \label{eq:EquivalenceTransformation}
    T = \id + \sum_{r=1}^\infty \lambda^r T_r
    \quad
    \textrm{with}
    \quad
    T_r \in \mathcal{D}
\end{equation}
such that $T$ is a module isomorphism, i.e.\  for all $a \in
\mathcal{A}$
\begin{equation}
    \label{eq:TIsEquivalence}
    T \circ \rho(a) = \tilde{\rho}(a) \circ T.
\end{equation}
Again, the order by order construction of $T$ gives an obstruction in
the Hochschild cohomology:
\begin{lemma}
    \label{lemma:ObstructionForEquivalence}
    Assume that $T^{(r)} = \id + \cdots + \lambda^r T_r$ is an
    equivalence between $\rho$ and $\tilde{\rho}$ up to order
    $\lambda^r$ such that $T_s \in \HCtype^0(\mathcal{A},
    \mathcal{D})$ for $s = 1, \ldots, r$. Then the condition for
    $T_{r+1} \in \HCtype^0(\mathcal{A}, \mathcal{D})$ to define an
    equivalence $T^{(r+1)} = T^{(r)} + \lambda^{r+1} T_{r+1}$ up to
    order $\lambda^{r+1}$ is
    \begin{equation}
        \label{eq:deltaTrplusEins}
        \delta T_{r+1} = E_r
    \end{equation}
    with $E_r \in \HCtype^1(\mathcal{A}, \mathcal{D})$ explicitly
    given by
    \begin{equation}
        \label{eq:ErExplicit}
        E_r(a) = 
        \sum_{s=0}^r
        \left(
            \tilde{\rho}_{r+1-s}(a) \circ T_s 
            - T_s \circ \rho_{r+1-s}(a)
        \right).
    \end{equation}
    Moreover, $\delta E_r = 0$ whence the recursive obstruction in
    order $\lambda^{r+1}$ is the class $[E_r] \in
    \HHtype^1(\mathcal{A}, \mathcal{D})$.
\end{lemma}
\begin{proof}
    Again, the only new thing compared to the purely algebraic
    situation is the simple observation that $E_r \in
    \HCtype^1(\mathcal{A}, \mathcal{D})$.
\end{proof}
Note that even if an obstruction occurs in higher orders, i.e.\ $[E_r]
\ne 0$ in $\HHtype^1(\mathcal{A}, \mathcal{D})$, the two module
deformations still might be equivalent as one is allowed to change the
already found $T_1, \ldots, T_r$. This makes the classification of
equivalences up to all orders very difficult in general. However, if
the first cohomology $\HHtype^1(\mathcal{A}, \mathcal{D}) = \{0\}$ is
trivial, the construction of $T$ can be done recursively and any two
deformations are equivalent.

Now, we consider one example for more specific types of cochains:
\begin{example}[Differential deformations]
    \label{example:DifferentialDeformations}
    Assume that $\mathcal{A}$ is commutative and also $\mathcal{E}$
    carries the additional structure of an associative, commutative
    algebra. Then we consider the (algebraic) differential
    operators
    \begin{equation}
        \label{eq:DiffopE}
        \mathcal{D} =
        \Diffop^\bullet(\mathcal{E})
        = \Union_{l=0}^\infty \Diffop^l (\mathcal{E})
    \end{equation}
    of the algebra $\mathcal{E}$. We assume for the undeformed module
    structure that $\rho_0(a)$ is a differential operator on
    $\mathcal{E}$ of order zero, i.e.\  $\rho_0(a) \in
    \Diffop^0(\mathcal{E})$ for all $a\in \mathcal{A}$. Since
    $\mathcal{A}$ is commutative, any right module is a left module
    and vice versa. For later use it will be convenient to deform
    $\mathcal{E}$ into a right module but use the $(\mathcal{A},
    \mathcal{A})$-bimodule structure
    \begin{equation}
        \label{eq:FunnyBimoduleStructure}
        a \cdot D \cdot b = \rho_0(a) \circ D \circ \rho_0(b)
    \end{equation}
    for the endomorphisms of $\mathcal{E}$, in contrast to
    \eqref{eq:DBimodule}. This will not affect the cohomological
    considerations but simplify some of the explicit formulas.  For
    `type' we choose the multi-differential cochains, i.e.
    \begin{equation}
        \label{eq:HCdiff}
        \HCdiff^k(\mathcal{A}, \mathcal{A})
        = \Union_{L \in \mathbb{N}_0^k} 
        \Diffop^L(\mathcal{A}, \ldots, \mathcal{A}; \mathcal{A}),
    \end{equation}
    where $L = (l_1, \ldots, l_k)$ is the multi-index denoting the
    multi-order of differentiation. Moreover, we consider
    \begin{equation}
        \label{eq:HCdiffDiffopE}
        \HCdiff^k (\mathcal{A}, \Diffop(\mathcal{E}))
        =
        \Union_{L \in \mathbb{N}^k_0} \Union_{l \in \mathbb{N}_0}
        \Diffop^L(\mathcal{A}, \ldots, \mathcal{A}; \Diffop^l(\mathcal{E})),
    \end{equation}
    where we use the \emph{left} module structure induced by $\rho_0$
    to specify multi-differential operators with values in
    $\Diffop^l(\mathcal{E})$ according to
    \eqref{eq:FunnyBimoduleStructure}. With the definition
    \eqref{eq:HCdiffDiffopE} a cochain $\phi \in
    \HCdiff^k(\mathcal{A}, \Diffop(\mathcal{E}))$ has the property
    that for any $a_1, \ldots, a_k \in \mathcal{A}$ the differential
    operator $\phi(a_1, \ldots, a_k)$ has some fixed order $l$
    independent of $a_1, \ldots, a_k$.  It is now easy to verify that
    $\HCdiff^\bullet(\mathcal{A}, \Diffop(\mathcal{E}))$ satisfies all
    requirements \textit{i.)} to \textit{iv.)}. Note that in general
    this is not true for $\Union_{L \in \mathbb{N}^k_0}
    \Diffop^L(\mathcal{A}, \ldots, \mathcal{A};
    \Diffop^\bullet(\mathcal{E}))$.
\end{example}

In the last part of our general considerations we focus on the
situation where some deformation $\rho$ exists (e.g.\ since the second
Hochschild cohomology is trivial) and where the first Hochschild
cohomology $\HHtype^1(\mathcal{A}, \mathcal{D})$ is trivial,
\begin{equation}
    \label{eq:HHtypeEinsNull}
    \HHtype^1(\mathcal{A}, \mathcal{D}) = \{0\}.
\end{equation}
Then we already know that all deformations are equivalent. We shall
now discuss the module endomorphisms of the deformed module. However,
we do not consider general module endomorphisms but only those which
are formal series of operators of the given type, i.e.\  in
$\mathcal{D}[[\lambda]]$. So for the undeformed situation the module
endomorphisms of interest are $\HHtype^0(\mathcal{A}, \mathcal{D}) =
\ker \delta \cap \mathcal{D} \subseteq \mathcal{D}$.  Clearly, they
form a subalgebra of $\mathcal{D}$ such that $\mathcal{E}$ becomes a
$(\HHtype^0(\mathcal{A}, \mathcal{D}), \mathcal{A})$-bimodule by the
very definition of module endomorphisms.

For abbreviation we denote the \emph{commutant} inside
$\mathcal{D}[[\lambda]]$, i.e.\  the module endomorphisms of the
\emph{deformed} module, by
\begin{equation}
    \label{eq:Kommutante}
    \mathcal{K} = 
    \left\{
        A \in \mathcal{D}[[\lambda]] 
        \; \big| \;
        A \circ \rho(a) = \rho(a) \circ A
        \; \textrm{for all} \; a \in \mathcal{A}[[\lambda]]
    \right\}.
\end{equation}
We will now make use of a complementary subspace
$\overline{\HHtype^0(\mathcal{A}, \mathcal{D})} \subseteq \mathcal{D}$
of the undeformed commutant, i.e.\  we choose
$\overline{\HHtype^0(\mathcal{A}, \mathcal{D})}$ such that
\begin{equation}
    \label{eq:Complement}
    \mathcal{D} 
    = \HHtype^0(\mathcal{A}, \mathcal{D}) 
    \oplus 
    \overline{\HHtype^0(\mathcal{A}, \mathcal{D})}
\end{equation}
which is always possible as we work over a field $\mathbb{K}$. Then
the following proposition describes the structure of $\mathcal{K}$:
\begin{proposition}
    \label{proposition:Kommutante}
    Every choice of a complementary subspace
    $\overline{\HHtype^0(\mathcal{A}, \mathcal{D})}$ induces a
    $\mathbb{K}$-linear map
    \begin{equation}
        \label{eq:rhoprime}
        \rho': \HHtype^0(\mathcal{A}, \mathcal{D})[[\lambda]]
        \longrightarrow
        \mathcal{D}[[\lambda]]
    \end{equation}
    of the form $\rho' = \id + \sum_{r=1}^\infty \lambda^r \rho'_r$
    with $\rho'_r: \HHtype^0(\mathcal{A}, \mathcal{D}) \longrightarrow
    \mathcal{D}$ and the following properties:
    \begin{enumerate}
    \item $\rho'_r(A) \in \overline{\HHtype^0(\mathcal{A},
          \mathcal{D})}$ for all $A \in \HHtype^0(\mathcal{A},
        \mathcal{D})$ and $r \ge 1$.
    \item $\rho'$ is a $\mathbb{K}[[\lambda]]$-linear bijection onto
        $\mathcal{K}$.
    \item $\rho'$ induces an associative deformation of the classical
        commutant
        \begin{equation}
            \label{eq:starprimeDef}
            \mu'(A, B)
            = \rho'^{-1}\left(\rho'(A) \circ \rho'(B)\right),
        \end{equation}
        where $A, B \in \HHtype^0(\mathcal{A},
        \mathcal{D})[[\lambda]]$.
    \item $\rho'$ defines a left module structure for the deformed
        algebra $(\HHtype^0(\mathcal{A}, \mathcal{D})[[\lambda]],
        \mu')$ on $\mathcal{E}[[\lambda]]$ such that
        $\mathcal{E}[[\lambda]]$ becomes a bimodule with respect to
        the two deformed algebras.
    \item Different choices of $\overline{\HHtype^0(\mathcal{A},
          \mathcal{D})}$ and $\rho$ yield equivalent deformations of
        $\HHtype^0(\mathcal{A}, \mathcal{D})$.
    \item Suppose in addition that $\HHtype^2(\mathcal{A},
        \mathcal{D}) = \{0\}$. Then we obtain a map
        \begin{equation}
            \label{eq:DefTheorien}
            \Def_{\mathrm{type}}(\mathcal{A}) \longrightarrow
            \Def(\HHtype^0(\mathcal{A}, \mathcal{D})),
        \end{equation}
        where $\Def$ denotes the set of equivalence classes of
        associative deformations.
    \end{enumerate}
\end{proposition}
\begin{proof}
    Let $A \in \HHtype^0(\mathcal{A}, \mathcal{D})$ be given. We
    construct $\rho'(A) \in \mathcal{K}$ recursively order by order.
    Obviously, in zeroth order $A \in \mathcal{K}$. Assume we have
    already constructed correction terms $\rho_1'(A)$, \ldots,
    $\rho_r'(A)$. Then it is easy to see that the error term in order
    $\lambda^{r+1}$ of $A + \lambda^1 \rho_1'(A) + \cdots + \lambda^r
    \rho_r'(A)$ to be in $\mathcal{K}$ is a $\delta$-closed cochain in
    $\HCtype^1(\mathcal{A}, \mathcal{D})$.  Thus it is a coboundary by
    assumption \eqref{eq:HHtypeEinsNull} and the splitting
    \eqref{eq:Complement} allows to choose a unique correction term in
    $\overline{\HHtype^0(\mathcal{A}, \mathcal{D})}$. By induction we
    obtain the first part and the injectivity of the second part.
    Conversely, if an element in the commutant $\mathcal{K}$ is given,
    then its lowest non-vanishing order is in $\HHtype^0(\mathcal{A},
    \mathcal{D})$ which can be `quantized' using $\rho'$. By a simple
    induction we obtain the surjectivity.  The third part is obvious
    as $\mathcal{K}$ is an associative algebra over
    $\mathbb{K}[[\lambda]]$ and $\rho'$ is the identity in zeroth
    order. The fourth part is clear by construction. For the last part
    we observe that $\mathcal{K}$ is independent of the choice of
    $\overline{\HHtype^0(\mathcal{A}, \mathcal{D})}$ whence it follows
    immediately that different choices of
    $\overline{\HHtype^0(\mathcal{A}, \mathcal{D})}$ give equivalent
    deformations.  Moreover, passing to a different module deformation
    $\tilde{\rho}$ we obtain an equivalence between $\rho$ and
    $\tilde{\rho}$ by $\HHtype^1(\mathcal{A}, \mathcal{D}) = \{0\}$.
    Thus, the commutants $\mathcal{K}$ are isomorphic, too, yielding
    an equivalence between $\rho'$ and $\tilde{\rho'}$. The last part
    is clear since for every deformation of $\mathcal{A}$ of the
    specified type we can deform the right module and hence
    $\HHtype^0(\mathcal{A}, \mathcal{D})$ in a unique way up to
    equivalence.
\end{proof}

%
%

\section{The topological bar and Koszul complex of $C^\infty(V)$}
\label{sec:homological}

In this section we recall some results on the homological algebra of
$C^\infty(V)$ with a convex open subset $V \subseteq \mathbb{R}^n$.
The main tools will be the topological bar resolution as well as the
topological Koszul resolution of $C^\infty(V)$. Most of the material
of this section is well-known and can be found in either
\cite{connes:1994a}, Sect.~III.2.$\alpha$, or
\cite{bordemann.et.al:2005a:pre}. Nevertheless, for later use we have
to present the results in some details.

For the following considerations we need the (topological) extended
algebra $\mathcal{A}^e$ which is given by
\begin{equation}
    \label{eq:extended-algebra}
    \mathcal{A}^e=C^\infty(V\times V).
\end{equation}
Note that in a purely algebraic context the extended algebra is just
$\mathcal{A} \otimes \mathcal{A}$ but here we use the completed tensor
product with respect to the canonical Fr\'{e}chet topology of smooth
functions, resulting in \eqref{eq:extended-algebra}.

For the definition of the bar complex we also consider the topological
version. One defines
\begin{equation}
    \label{eq:bar-complex}
    X_0 
    = \mathcal{A}^e=C^\infty(V\times V)
    \quad
    \textrm{and}
    \quad
    X_k = C^\infty(V\times V^{k} \times V)
\end{equation}
for $k \in \mathbb{N}$ with the $\mathcal{A}^e$-module structures
\begin{equation}
    \label{eq:module-structure-bar-complex}
    (\hat{a}\chi)(v,q_1,\ldots,q_k,w)=\hat{a}(v,w)\chi(v,q_1,\ldots,q_k,w)
\end{equation}
for $\hat{a}\in \mathcal{A}^e$, $\chi\in X_k$ and $v, w, q_1, \ldots,
q_k \in V$.  The bar complex $(X_\bullet,\partial_X)$ and the
corresponding bar resolution over $\mathcal{A}$ are then given by the
exact sequence
\begin{equation}
    \label{eq:diagram-bar-resolution}
      \xymatrix{{0} 
        & \mathcal{A} \ar[l] 
        & X_0 \ar[l]_\epsilon 
        & X_1 \ar[l]_{\partial_X^1} 
        & \ldots \ar[l]_{\partial_X^2} 
        & X_k \ar[l]_{\partial_X^k} 
        & \ldots \ar[l]_{\partial_X^{k+1}} 
      },
\end{equation}
where the boundary operators $\partial_X^k$ and the augmentation
$\epsilon$ are defined by
\begin{align}
    &(\partial_X^k \chi)(v,q_1,\ldots,q_{k-1},w)
    =
    \chi(v ,v,q_1,\ldots,q_{k-1},w)
    \nonumber\\
    &\quad 
    +\sum_{i=1}^{k-1}
    (-1)^i\chi(v,q_1,\ldots,q_i,q_i,\ldots,q_{k-1},w)
    +(-1)^k \chi(v,q_1,\ldots,q_{k-1},w,w)
    \label{eq:boundary-operator-bar-complex}
\end{align}
and
\begin{equation}
    \label{eq:augmentation}
    (\epsilon\hat{a})(v) =  \hat{a}(v, v).
\end{equation}
It is well known that $\partial_X^k$ and $\epsilon$ are homomorphisms
of $\mathcal{A}^e$-modules and $\partial_X^{k-1}\circ \partial_X^k=0$
for all $k\ge 2$ and $\epsilon\circ \partial_X^1=0$.  The well-known
homotopies $h_X^{-1}: \mathcal{A}\longrightarrow X_0$ and $h_X^k:
X_k\longrightarrow X_{k+1}$ are given by
\begin{equation}
    \label{eq:homotopybarcomplex}
    (h_X^{-1}a)(v,w) = a(v) 
    \quad
    \textrm{and}
    \quad
    (h_X^k \chi)(v,q_1, \ldots, q_{k+1},w)
    = (-1)^{k+1}\chi(v,q_1,\ldots,q_{k+1})
\end{equation}
for $k \ge 0$. Then one has
\begin{eqnarray}
    \epsilon \circ h_X^{-1}&=&\id_\mathcal{A}, \nonumber\\
    \label{eq:homotopy-bar-complex}
    h_X^{-1}\circ \epsilon +\partial_X^1 \circ h_X^0 &=& \id_{X_0}
    \quad \textrm{and}\\
    h_X^{k-1} \circ \partial_X^k +\partial_X^{k+1} \circ h_X^k &=&
    \id_{X_k} \quad \forall k\ge 1. \nonumber
\end{eqnarray}
Hence the sequence \eqref{eq:diagram-bar-resolution} is exact and thus
defines a resolution of $\mathcal{A}$. Note that the modules $X_k$ are
topologically free as $\mathcal{A}^e$-modules, confer
\cite{connes:1994a} for a more general version of this.

For the (topological) Koszul complex one considers the (topologically
free) $\mathcal{A}^e$-modules
\begin{equation}
    \label{eq:Koszul-complex-k}
    K_0  
    = \mathcal{A}^e
    \quad
    \textrm{and}
    \quad
    K_k  = 
    \mathcal{A}^e\otimes_{\mathbb{R}}
    \Anti^k(\mathbb{R}^n)^* 
    \cong C^\infty(V\times V, \Anti^k(\mathbb{R}^n)^*).
\end{equation}
With a basis $\{e_i\}_{i = 1, \ldots, n}$ of $\mathbb{R}^n$ and the
corresponding dual basis $\{e^i\}_{i = 1, \ldots, n}$ of
$(\mathbb{R}^n)^*$ the elements $\omega \in K_k$ can be written as
\begin{equation*}
    \omega
    = \frac{1}{k!} \sum_{i_1,\ldots,i_k=1}^n 
    \omega_{i_1\ldots i_k} e^{i_1} \wedge \ldots \wedge e^{i_k}
\end{equation*}
with $\omega_{i_1 \ldots i_k}\in \mathcal{A}^e$.  The Koszul complex
$(K_\bullet, \partial_K)$ and the corresponding \emph{finite}
resolution of $\mathcal{A}$ are given by
\begin{equation}
    \label{eq:diagram-Koszul-resolution}
     \xymatrix{{0} 
       & \mathcal{A} \ar[l] 
       & K_0 \ar[l]_\epsilon 
       & K_1 \ar[l]_{\partial_K^1} 
       & \ldots \ar[l]_{\partial_K^2} 
       & K_k \ar[l]_{\partial_K^k} 
       & \ldots \ar[l]_{\partial_K^{k+1}} 
       & K_n \ar[l]_{\partial_K^{n}}
       & \ar[l] 0
      },
\end{equation}
where the augmentation $\epsilon: \mathcal{A}^e\longrightarrow
\mathcal{A}$ is the same as in \eqref{eq:augmentation} and the
boundary operators $\partial_K^k$ are defined by
\begin{equation}
    \label{eq:boundary-operator-Koszul-complex}
    \left((\partial_K^k\omega)(v, w)\right)(x_1, \ldots, x_{k-1})
    =
    \left(\omega(v, w)\right)(v-w, x_1, \ldots, x_{k-1}) 
\end{equation}
for $v, w \in V$ and $x_1, \ldots, x_{k-1} \in \mathbb{R}^n$.  Again,
the maps $\partial_K^k$ are $\mathcal{A}^e$-module homomorphisms with
$\partial_K^{k-1} \circ \partial_K^k = 0 $ for all $k\ge 2$ and
$\epsilon \circ \partial_K^1 = 0$. The maps $h_K^{-1} = h_X^{-1}$ and
$h_K^k: K_k \longrightarrow K_{k+1}$ with
\begin{equation}
    \label{eq:homotopy-map-Koszul-complex-k}
    (h_K^k\omega)(v, w)
    = -e^j \wedge \int_0^1 t^k 
    \frac{\partial \omega}{\partial w^j} (v, tw + (1-t)v) \D t
\end{equation}
for $k \ge 0$ yield the homotopy identities
\begin{eqnarray}
    \epsilon \circ h_K^{-1}&=&\id_\mathcal{A}, \nonumber\\
    \label{eq:homotopy-Koszul-complex}
    h_K^{-1}\circ \epsilon +\partial_K^1 \circ h_K^0 &=& \id_{K_0}
    \quad \textrm{and}\\
    h_K^{k-1} \circ \partial_K^k +\partial_K^{k+1} \circ h_K^k &=&
    \id_{K_k} \quad \forall k\ge 1. \nonumber
\end{eqnarray}
Hence \eqref{eq:diagram-Koszul-resolution} is indeed a topologically
free resolution of $\mathcal{A}$. Note that in
\eqref{eq:homotopy-map-Koszul-complex-k} we made use of the convexity
of $V$.

For all $k\ge 0$ we consider the maps $F^k: K_k\longrightarrow X_k$
from \cite{connes:1994a} defined by
\begin{equation}
    \label{eq:chain-maps-Koszul-bar}
    (F^k \omega)(v, q_1, \ldots, q_k, w) 
    = \left( 
        \omega(v, w)
    \right)(q_1 - v, \ldots, q_k - v)
\end{equation}
and the maps $G^k: X_k \longrightarrow K_k$ from
\cite{bordemann.et.al:2005a:pre} defined by
\begin{equation}
    \label{eq:chain-maps-bar-Koszul}   
    \begin{split}
        (G^k \chi)(v, w)
        &= \sum_{i_1, \ldots, i_k = 1}^n 
        e^{i_1} \wedge \ldots \wedge e^{i_k} 
        \int_0^1 \int_0^{t_1} \cdots \int_0^{t_{k-1}} \\
        &\quad
        \frac{\partial^k\chi}
        {\partial q_1^{i_1} \cdots \partial q_k^{i_k}}
        (v, t_1 v + (1 - t_1) w, \ldots, t_k v + (1 - t_k) w, w)
        \D t_1 \cdots \D t_k.
    \end{split}
\end{equation}
In particular, $F^0 = \id_{\mathcal{A}^e} = G^0$.  Note that $G^k$ is
well-defined since we assume $V$ to be convex.  Clearly, these maps
are $\mathcal{A}^e$-module homomorphisms and it is a straightforward
computation to prove that $F^k$ and $G^k$ are chain maps, see
\cite{connes:1994a, bordemann.et.al:2005a:pre}.  This means that for
all $k \ge 0$
\begin{equation}
    \label{eq:chain-maps}
    F^k\circ \partial_K^{k+1} = \partial_X^{k+1} \circ F^{k+1} 
    \quad
    \textrm{and}
    \quad 
    G^k \circ \partial_X^{k+1} = \partial_K^{k+1} \circ G^{k+1}.
\end{equation}
Thus we have the commutative diagram of $\mathcal{A}^e$-module
morphisms
\begin{equation}
    \label{eq:diagram-bar-Koszul}
    \xymatrix{{0} 
      & \mathcal{A} \ar[l] \ar @{=} [d] 
      & X_0 \ar @{=} [d]_{\id_{\mathcal{A}^e}} \ar[l]_\epsilon 
      & X_1 \ar[l]_{\partial_X^1} \ar @<+.5ex>[d]^{G^1} 
      & \ldots \ar[l]_{\partial_X^2} 
      & X_k \ar[l]_{\partial_X^k} \ar @<+.5ex>[d]^{G^k} 
      & X_{k+1} \ar[l]_{\partial_X^{k+1}} \ar @<+.5ex>[d]^{G^{k+1}} 
      & \ldots \ar[l]_{\partial_X^{k+2}} \\
      {0} 
      & \mathcal{A} \ar[l] 
      & K_0 \ar[l]^\epsilon 
      & K_1 \ar[l]^{\partial_K^1} \ar @<+.5ex>[u]^{F^1} 
      & \ldots \ar[l]^{\partial_K^2} 
      & K_k \ar[l]^{\partial_K^k} \ar @<+.5ex>[u]^{F^k} 
      & K_{k+1} \ar[l]^{\partial_K^{k+1}} \ar @<+.5ex>[u]^{F^{k+1}} 
      & \ldots.  \ar[l]^{\partial_K^{k+2}} 
    }
\end{equation}
Another direct computation shows that
\begin{equation}
    \label{eq:chain-maps-id}
    G^k\circ F^k=\id_{K_k} \quad \forall k\ge0
\end{equation}
and thus it is clear that
\begin{equation}
    \label{eq:chain-maps-projection}
    \Theta^k= F^k\circ G^k: X_k\longrightarrow X_k 
\end{equation}
for all $k\ge 0$ is a projection $\Theta^k\circ \Theta^k=\Theta^k$ on
the bar complex with $\Theta^k\circ \partial_X^{k+1}=
\partial_X^{k+1}\circ \Theta^{k+1}$. Clearly, $\Theta^0=\id_{X_0}$.
Applying $\Theta^k$ to the last equation of
\eqref{eq:homotopy-bar-complex} leads to
\begin{equation}
    \label{eq:homotopy-Theta}
    \Theta^k\circ h_X^{k-1}\circ \partial_X^k +
    \partial_X^{k+1} \circ \Theta^{k+1} \circ h_X^k= \Theta^k.
\end{equation}
Subtraction of \eqref{eq:homotopy-Theta} from
\eqref{eq:homotopy-bar-complex} shows that the chain homomorphisms
$\id_{X}$ and $\Theta$ are homotopic since
\begin{equation}
    \label{eq:id-Theta-homotopic}
    \id_{X_k}-\Theta^k= s^{k-1}\circ \partial _X^k + \partial_X^{k+1}
    \circ s^k \quad \forall k\ge 1
\end{equation}
with
\begin{equation}
    \label{eq:map-for-id-Theta-homotopic}
    s^{k-1} = (\id_{X_k} - \Theta^k)\circ h_X^{k-1}:
    X_{k-1}\longrightarrow X_k.
\end{equation}
Note that in \cite{bordemann.et.al:2005a:pre} the homotopy formula
\eqref{eq:id-Theta-homotopic} was obtained by a recursive and less
explicit definition of $s^k$.
\begin{remark}
    \label{remark:Aufloesung}
    From the point of view of topologically projective (even free in
    our case) resolutions, the explicit construction of $F$ and $G$
    and the homotopies $s$ is obsolete: this follows by abstract
    nonsense arguments. However, later on we are interested in
    Hochschild cochains which have additional properties beside being
    continuous. For this refined notion we need to prove by hand that
    the maps $F$, $G$, $\Theta$, and $s$ are compatible with these
    additional requirements whence we need the \emph{explicit}
    formulas.
\end{remark}

From now on, we closely follow \cite{bordemann.et.al:2005a:pre}. Let
$\mathcal{M}$ be a Hausdorff and complete topological
$\mathcal{A}$-left module with respect to the Fr{\'e}chet topology of
$\mathcal{A} = C^\infty(V)$, i.e.\ the bilinear multiplication
\begin{equation}
    \label{eq:left-module-structure}
    (\mathcal{A}\times \mathcal{M})
    \ni (a,m) \mapsto a\cdot m \in \mathcal{M}
\end{equation} 
is continuous.  Furthermore, we demand that $\mathcal{M}$ has an
$\mathcal{A}$-right module structure such that there exists an $l \in
\mathbb{N}$ such that the right module multiplication can be expressed
in terms of the left module multiplication by
\begin{equation}
    \label{eq:right-module-structure-differential}
    m \cdot b
    = \sum_{|\beta|\le l}
    \frac{\partial^{|\beta|}b} {\partial v^\beta} \cdot m^\beta 
    \quad \forall b \in \mathcal{A}, m \in \mathcal{M}
\end{equation}
with elements $m^\beta\in \mathcal{M}$ depending continuously on $m$.
As $l$ is uniform for $\mathcal{M}$ it follows that the trilinear map
$(a, m, b) \mapsto a\cdot m\cdot b$ is continuous, too. Thus
$\mathcal{M}$ is a topological \emph{bimodule}. By continuity,
\begin{equation}
    \label{eq:regularity-condition}
    (a\otimes b) \cdot m 
    = a\cdot m \cdot b 
    \quad \forall a, b \in \mathcal{A}, m\in \mathcal{M}
\end{equation}
extends to a unique $\mathcal{A}^e$-left-module structure of
$\mathcal{M}$ which is explicitly given by
\begin{equation}
    \label{eq:extended-structure-and-left-module-structure}
    \hat{a} \cdot m
    = \sum_{|\beta|\le l} 
    \left( 
        v \mapsto 
        \left.
            \frac{\partial^{|\beta|}\hat{a}}
            {\partial w^\beta}(v, w)
        \right|_{w=v}
    \right) \cdot m^\beta
    \quad \forall \hat{a}\in \mathcal{A}^e.
\end{equation}

For the rest of this section we use a definition of differential maps,
which is slightly different to the purely algebraic definition.
\begin{definition}
    \label{definition:differential-maps}
    Let $k \in \mathbb{N}$. An $\mathbb{R}$-multi-linear map $\phi:
    \mathcal{A}\times \ldots \times \mathcal{A} \longrightarrow
    \mathcal{M}$ with $k$ arguments is said to be differential of
    multi-order $L=(l_1, \ldots, l_k) \in \mathbb{N}_0^k$, if it has
    the form
    \begin{equation}
        \label{eq:differential-map}
        \phi(a_1, \ldots, a_k)
        = 
        \sum_{|\alpha_1| \le l_1,\ldots,|\alpha_k| \le l_k}
        \left( 
            \frac{\partial^{|\alpha_1|} a_1}{\partial v^{\alpha_1}}
            \cdots
            \frac{\partial^{|\alpha_k|} a_k}{\partial v^{\alpha_k}}
        \right)
        \cdot \phi^{\alpha_1\cdots \alpha_k}
    \end{equation}
    with multi-indices $\alpha_1, \ldots, \alpha_k \in \mathbb{N}_0^n$
    and $\phi^{\alpha_1\cdots \alpha_k} \in \mathcal{M}$.  The
    $L$-differential maps are denoted by $\Diffop^L(\mathcal{A},
    \mathcal{M})$.
\end{definition}
\begin{remark}
    \begin{enumerate}
    \item Clearly, any differential map in the sense of
        Definition~\ref{definition:differential-maps} is continuous.
    \item For many examples of modules $\mathcal{M}$, like the one of
        interest in the present paper,
        Definition~\ref{definition:differential-maps} is consistent
        with the purely algebraic definition of multi-differential
        operators.
    \end{enumerate}
\end{remark}
In this sense, \eqref{eq:right-module-structure-differential} means
that the right module structure is differential with respect to the
left one, so $\phi_m: b \mapsto m \cdot b$ is a differential operator
depending on $m$ of order $l$.

For $k\ge 0$ we now consider the vector space
$\Hom_{\mathcal{A}^e}^{\mathrm{cont}} (X_k,\mathcal{M})$ of all
$\mathcal{A}^e$-linear and \emph{continuous} maps, which is both an
$\mathcal{A}$- and an $\mathcal{A}^e$-module. Using the pull-backs
$\delta_X^{k-1} = (\partial_X^k)^*$ it is clear that
$(\Hom_{\mathcal{A}^e}^{\mathrm{cont}} (X_\bullet,\mathcal{M}),
\delta_X)$ is a well-defined complex.  With the above results the
Hochschild complex
$(\mathrm{HC}^\bullet_{\mathrm{cont}}(\mathcal{A},\mathcal{M}),
\delta)$ of \emph{continuous cochains} is well-defined, too. By
continuity, the maps $\Xi^k: \Hom_{\mathcal{A}^e}^{\mathrm{cont}}(X_k,
\mathcal{M}) \longrightarrow
\mathrm{HC}^k_{\mathrm{cont}}(\mathcal{A},\mathcal{M})$ defined by
\begin{align}
    \label{eq:bijection-Hom-bar-Hochschild-complex-explicitly}
    \left( \Xi^k(\psi) \right) (a_1, \ldots, a_k)
    = \psi(1\otimes a_1 \otimes \ldots \otimes a_k \otimes 1)
\end{align}
are bijective. Additionally, it can be shown that 
\begin{equation}
    \label{eq:chain-homomorphism-Hom-bar-Hochschild-complex}
    \Xi \circ \delta_X = \delta \circ \Xi,
\end{equation}
so $\Xi$ is even an isomorphism of complexes.
Definition~\ref{definition:differential-maps} and
\eqref{eq:right-module-structure-differential} assure that the
continuous Hochschild complex contains the subcomplex of differential
cochains in the sense of
Definition~\ref{definition:differential-maps}. In this section we will
denote this subcomplex by
$(\mathrm{HC}^\bullet_{\mathrm{diff}}(\mathcal{A},\mathcal{M}),\delta)$
with a slight abuse of notation. In the next section we will
specialize to a particular bimodule $\mathcal{M}$ for which this
coincides with Example~\ref{example:DifferentialDeformations}.  This
and \eqref{eq:bijection-Hom-bar-Hochschild-complex-explicitly}
motivate the following definition.
\begin{definition} 
    \label{definition:DifferentialHomXk}
    Let $k\in \mathbb{N}_0$. An element $\psi \in \Hom_{\mathcal{A}^e}
    (X_k,\mathcal{M})$ is said to be differential of multi-order $L =
    (l_1, \ldots, l_k) \in \mathbb{N}_0^k$, if it has the form
    \begin{equation}
        \label{eq:Hom-modules-elements-differential}
        \psi =
        \sum_{
          \begin{subarray}{c}
              |\alpha_1|\le l_1,\ldots,|\alpha_k|\le l_k\\
              |\beta|\le l
          \end{subarray}
        }
        \left(
            v\mapsto \left.\frac{\partial^{|\alpha_1|+ \cdots + |\alpha_k|
                    +|\beta| }}
                {\partial q_1^{\alpha_1} \cdots \partial
                  q_k^{\alpha_k} \partial w^\beta}
            \right|_{
              \begin{subarray}{c}
                  q_1=\ldots=q_k=v\\
                  w=v
              \end{subarray}} 
            (\cdot) \right)
        \cdot \psi^{\alpha_1 \cdots \alpha_k\beta}
    \end{equation}
    with multi-indices $\alpha_1, \ldots, \alpha_k, \beta \in
    \mathbb{N}_0^n$ and $\psi^{\alpha_1 \cdots \alpha_k \beta}\in
    \mathcal{M}$.  The differential elements of multi-order $L$ are
    denoted by $\Hom_{\mathcal{A}^e}^{\mathrm{diff},L}
    (X_k,\mathcal{M})$.
\end{definition}
Equation \eqref{eq:Hom-modules-elements-differential} means that under
the map $\psi$ an element $\chi \in X_k$ first gets differentiated and
then evaluated at $q_1 = \ldots = q_k = w = v$. Finally, this
expression, seen as a function of $v \in V$ and therefore as an
element in $\mathcal{A}$, is multiplied with elements of $\mathcal{M}$
from the left.  A direct computation shows that
$\Hom_{\mathcal{A}^e}^{\mathrm{diff}} (X_\bullet, \mathcal{M})$ is a
subcomplex of $\Hom_{\mathcal{A}^e}^{\mathrm{cont}} (X_\bullet,
\mathcal{M})$.  By the very construction,
\eqref{eq:bijection-Hom-bar-Hochschild-complex-explicitly} restricts
to an isomorphism of complexes
\begin{equation}
    \label{eq:bijection-Hom-bar-Hochschild-diff}
    \Xi: \left(
        \Hom_{\mathcal{A}^e}^{\mathrm{diff}} (X_\bullet, \mathcal{M}),
        \delta_X
    \right)
    \longrightarrow
    \left(
        \mathrm{HC}^\bullet_{\mathrm{diff}} (\mathcal{A},\mathcal{M}),
        \delta
    \right).
\end{equation}

In a last step we consider the complex $(\Hom_{\mathcal{A}^e}
(K_\bullet,\mathcal{M}), \delta_K)$ with $\delta_K^{k-1} =
(\partial_K^k)^*$ and find the following important proposition.
\begin{proposition}
    \label{proposition:restrictions-diff}
    Let $k\in \mathbb{N}_0$. The pull-backs
    \begin{equation}
        \label{eq:differential-G-star}
        (G^k)^*: \Hom_{\mathcal{A}^e} (K_k,\mathcal{M})
        \longrightarrow \Hom_{\mathcal{A}^e}^{\mathrm{diff},L}
        (X_k,\mathcal{M})
    \end{equation}
    only take values in the differential cochains of multi-order $L =
    (l+1, \ldots, l+1)$. With the same multi-index $L$ we have
    \begin{equation}
        \label{eq:differential-Theta-star}
        (\Theta^k)^*:
        \Hom_{\mathcal{A}^e}^{\mathrm{diff}}(X_k, \mathcal{M})
        \longrightarrow
        \Hom_{\mathcal{A}^e}^{\mathrm{diff},L}(X_k, \mathcal{M}).
    \end{equation}
    Further, the pull-backs $(h_X^k)^*$ restrict to the
    differential complex in such a way that
    \begin{equation}
        \label{eq:differential-h-star}
        (h_X^{k})^*:
        \Hom_{\mathcal{A}^e}^{\mathrm{diff},L}(X_{k+1}, \mathcal{M})
        \longrightarrow
        \Hom_{\mathcal{A}^e}^{\mathrm{diff},\tilde{L}}(X_k,\mathcal{M}) 
    \end{equation}
    for all $L = (l_1, \ldots, l_{k+1}) \in \mathbb{N}_0^{k+1}$ with
    $\tilde{L}=(l_1,\ldots,l_k)$.
\end{proposition}
\begin{proof}
    All assertions follow from easy computations and an explicit
    counting of the orders of differentiation.
\end{proof}

With the corresponding pull-backs $(F^k)^*$ we get the
commutative diagram
\begin{equation}
    \label{eq:diagram-Hom-bar-Koszul}
    \xymatrix{
      \ldots \ar[rr]^(.35){\delta_X^{k-1}} &&
      \Hom_{\mathcal{A}^e}^{\mathrm{diff}} (X_k,\mathcal{M}) \ar
      @<-.5ex> [d]_{(F^k)^*}
      \ar [rr]^{\delta_X^k}&&
      \Hom_{\mathcal{A}^e}^{\mathrm{diff}} (X_{k+1},\mathcal{M}) \ar 
      @<-.5ex> [d]_{(F^{k+1})^*}
      \ar[rr]^(.65){\delta_X^{k+1}}&&
      \ldots\\
      \ldots \ar[rr]^(.35){\delta_X^{k-1}} &&
      \Hom_{\mathcal{A}^e} (K_k,\mathcal{M}) \ar @<-.5ex>
      [u]_{(G^k)^*} 
      \ar [rr]^{\delta_K^k}&&
      \Hom_{\mathcal{A}^e} (K_{k+1},\mathcal{M}) \ar @<-.5ex>
      [u]_{(G^{k+1})^*} 
      \ar[rr]^(.65){\delta_K^{k+1}}&&
      \ldots.
    }
\end{equation}
With Proposition \ref{proposition:restrictions-diff} we thus have a
corresponding equation to \eqref{eq:id-Theta-homotopic} and
consequently the complexes in \eqref{eq:diagram-Hom-bar-Koszul} have
the same cohomologies.  Together with the isomorphism
\eqref{eq:bijection-Hom-bar-Hochschild-diff}, we get the isomorphisms
\begin{equation}
    \label{eq:equivalence-Hochschild-Hom-Koszul-cohomology}
    \mathrm{HH}^\bullet_{\mathrm{diff}}(\mathcal{A},\mathcal{M})\cong
    \mathrm{H}(\Hom_{\mathcal{A}^e}^{\mathrm{diff}}
    (X_\bullet,\mathcal{M})) \cong
    \mathrm{H}(\Hom_{\mathcal{A}^e} (K_\bullet,\mathcal{M}))
\end{equation}
for the cohomology of the differential Hochschild complex.  Note that
every isomorphism in
\eqref{eq:equivalence-Hochschild-Hom-Koszul-cohomology} is induced by
explicitly given maps on the level of cochains. Concerning our further
application of these results, we need the following obvious
generalization.
\begin{proposition}
    \label{proposition:isomorphism-Hochschild-bar}
    Let $\mathcal{M}^\bullet = \bigcup_{l=0}^\infty \mathcal{M}^l$ be
    a filtered $\mathcal{A}$-module, i.e.\  $\mathcal{M}^l\subset
    \mathcal{M}^{l+1}$ and $\mathcal{A} \cdot \mathcal{M}^l\subset
    \mathcal{M}^l$ for all $l\in \mathbb{N}$, such that every
    $\mathcal{M}^l$ satisfies the properties above. Moreover, the
    topologies have to respect the filtration, which means that for
    all $l \in \mathbb{N}$ the topology of $\mathcal{M}^l$ is given by
    the induced one from $\mathcal{M}^{l+1}$.  Then we have:
    \begin{enumerate}
    \item The unions
        \begin{equation}
            \label{eq:extended-complexes}
            \left(\Union_{l=0}^\infty \mathrm{HC}^\bullet_{\mathrm{diff}}
                (\mathcal{A},\mathcal{M}^l),\delta \right),
            \left(\Union_{l=0}^\infty
                \Hom_{\mathcal{A}^e}^{\mathrm{diff}} (X_\bullet,
                \mathcal{M}^l), \delta_X \right), \: \textrm{and} \:
            \left(\Union_{l=0}^\infty
                \Hom_{\mathcal{A}^e}(K_\bullet, \mathcal{M}^l),\delta_K \right)
        \end{equation}
        are sub-complexes of
        $\HC^\bullet(\mathcal{A},\mathcal{M}^\bullet)$,
        $\Hom_{\mathcal{A}^e}^{\mathrm{cont}} (X_\bullet,
        \mathcal{M}^\bullet)$, and $\Hom_{\mathcal{A}^e}(K_\bullet,
        \mathcal{M}^\bullet)$, respectively.
    \item The isomorphisms
        \eqref{eq:bijection-Hom-bar-Hochschild-diff} for each $l$
        induce an isomorphism of complexes
        \begin{equation}
            \label{eq:isomorphism-complex}
            \Xi: 
            \left(
                \Union_{l=0}^\infty
                \Hom_{\mathcal{A}^e}^{\mathrm{diff}}
                (X_\bullet, \mathcal{M}^l), 
                \delta_X
            \right)  
            \longrightarrow
            \left(
                \Union_{l=0}^\infty 
                \mathrm{HC}^\bullet_{\mathrm{diff}} 
                (\mathcal{A}, \mathcal{M}^l),
                \delta 
            \right). 
        \end{equation}
    \item The pull-backs $(G^k)^*, (F^k)^*, (\Theta^k)^*, (h_X^k)^*$
        and $(s^k)^*$ naturally extend to the complexes
        \eqref{eq:extended-complexes}. Thus, we have induced
        isomorphisms for the corresponding cohomologies.
    \end{enumerate}
\end{proposition}

%
%

\section{Deformation Quantization of Surjective Submersions}
\label{sec:SurjectiveSubmersions}

In this section we prove Theorem~\ref{theorem:def-quant-surj-sub}.
Thus let $\mathsf{p}: P \longrightarrow M$ be a surjective submersion
with total space $P$ and basis $M$ of dimension $n$. It is easy to see
that the smooth functions $C^\infty(M)$ and $C^\infty(P)$ endowed with
the point-wise product and the right module structure
\begin{equation}
    \label{eq:right-module-structure-surjective-submersion}
    \rho_0 (f,a) = f\cdot  \mathsf{p}^*a 
    \quad
    \forall f \in C^\infty(P), a \in C^\infty(M)
\end{equation}
satisfy all assumptions of Example~\ref{example:DifferentialDeformations}.
\begin{remark}
    For any manifold $M$ and any $C^\infty(M)$-(bi)module
    $\mathcal{M}$ we use the abbreviations
    \begin{equation}
        \label{eq:abbreviation-Diffop-P}
        \Diffop^\bullet(M) 
        = \Diffop^\bullet
        (C^\infty(M); C^\infty(M)),
    \end{equation}
    \begin{equation}
        \label{eq:abbreviation-Diffop-M}        
        \Diffop^L(M, \mathcal{M}) = 
        \Diffop^L (\underbrace{C^\infty(M), \ldots,
          C^\infty(M)}_{k\textrm{-times}}; \mathcal{M})  
    \end{equation}
    for $L = (l_1, \ldots, l_k)\in \mathbb{N}_0^k,$ $k\in \mathbb{N}$.
    Moreover, we set
    \begin{equation}
        \label{eq:abbreviation-Hochschild}
        \mathrm{HC}^\bullet(M,\mathcal{M})
        = \mathrm{HC}^\bullet(C^\infty(M), \mathcal{M}).
    \end{equation}
\end{remark}

As we are interested in differential deformations of
\eqref{eq:right-module-structure-surjective-submersion}, we have to
consider the Hochschild complex $\mathrm{HC}^\bullet_\mathrm{diff} (M,
\Diffop(P))$ as in Example~\ref{example:DifferentialDeformations}. A
straightforward generalization of the well-known considerations for
ordinary differential operators, see e.g.\ \cite{waldmann:2007a},
App.~A.5, shows that in the present situation we have the expected
local expressions.
\begin{lemma}
    \label{lemma:local-form-of-Diffops}
    Let $(U,x)$ be a local chart of $M$ and $\phi\in \Diffop^L(M,
    \Diffop^\bullet (P))$ with multi-order of differentiation $L =
    (l_1, \ldots, l_k)\in \mathbb{N}_0^k$. Then there exist uniquely
    defined differential operators $\phi_U^{\alpha_1 \cdots
      \alpha_k}\in \Diffop^\bullet (\mathsf{p}^{-1}(U))$ with
    multi-indices $\alpha_i\in \mathbb{N}_0^n$, $i=1,\ldots, k$, such
    that
    \begin{equation}
        \label{eq:local-form-of-Diffop}
        \phi(a_1,\ldots,a_k)|_{\mathsf{p}^{-1}(U)}= \sum_{|\alpha_1|\le l_1,\ldots,
          |\alpha_k|\le l_k}
        \left(\frac{\partial^{|\alpha_1|}a_1}{\partial x^{\alpha_1}}
            \cdots 
            \frac{\partial^{|\alpha_k|}a_k}{\partial
              x^{\alpha_k}}\right) \cdot \phi_U^{\alpha_1 \cdots \alpha_k}
    \end{equation}
    for all $a_1,\ldots, a_k\in C^\infty(M)$, where
    $|_{\mathsf{p}^{-1}(U)}$ denotes the restriction of differential
    operators.
\end{lemma}
\begin{remark}
    \label{remark:local-form-of-Diffops}
    It is easy to check that Lemma~\ref{lemma:local-form-of-Diffops}
    even holds for $\phi \in \Diffop^L(M, \Diffop^l (P))$ with a fixed
    $l\in \mathbb{N}_0$.  Then $\phi_U^{\alpha_1 \cdots \alpha_k}\in
    \Diffop^l (\mathsf{p}^{-1}(U))$.
\end{remark}
\begin{lemma}
    \label{lemma:Hochschild-complex-simple}
    Let $\mathsf{p}: P \longrightarrow M$ be a surjective submersion.
    Then for all $k \in \mathbb{N}$
    \begin{equation}
        \label{eq:Hochschild-complex-simple}
        \HCdiff^k (M, \Diffop(P))
        =
        \Union_{L \in \mathbb{N}_0^k}
        \Diffop^L (M, \Diffop^\bullet(P)).
    \end{equation}
\end{lemma}
\begin{proof}
    By \eqref{eq:HCdiffDiffopE} we have to prove that
    $\Diffop^L(M, \Diffop(P)) \subseteq \bigcup_{l=0}^\infty
    \Diffop^L(M, \Diffop^l(P))$ for all $L\in \mathbb{N}_0^k$.
    Locally, this follows from \eqref{eq:local-form-of-Diffop}, since
    the degree $l$ of the restricted operator
    $\phi(a_1,\ldots,a_k)|_{U}\in \Diffop^l (\mathsf{p}^{-1}(U))$ is
    limited by the degrees of the finitely many $\phi_U^{\alpha_1
      \cdots \alpha_k}$ which do not depend on $a_1, \ldots, a_k$.
    Globally, one shows by appropriately chosen $a_1, \ldots, a_k$
    that there is a uniform bound on the locally defined degrees $l$
    according to \eqref{eq:local-form-of-Diffop}.
\end{proof}

The strategy to compute the differential Hochschild cohomologies
includes three steps. First we observe that the Hochschild complex can
be localized. In a second step we consider a surjective submersion
$\mathsf{pr}_1: V \times G\longrightarrow V$ with a manifold $G$ and
an open convex subset $V \subseteq \mathbb{R}^n$.  For that case we
will be able to compute the Hochschild cohomologies using the
techniques developed in Section~\ref{sec:homological}.  Using a
partition of unity, we show that this result is sufficient to compute
the Hochschild cohomologies we started with.

By the constant rank theorem for any point $u \in P$ there exist open
subsets $\tilde{U} \subset P$ with $u \in \tilde{U}$ and $U \subset M$
with $p = \mathsf{p} (u) \in U$ together with diffeomorphisms $x: U
\longrightarrow V \subset \mathbb{R}^n$ and $\tilde{x}: \tilde{U}
\longrightarrow V \times G \subset \mathbb{R}^{n+k}$, such that
$\mathsf{p} (\tilde{U}) = U$ and $x \circ \mathsf{p} \circ
\tilde{x}^{-1} = \mathsf{pr}_1$. Clearly, in this case
$G\subset\mathbb{R}^k$ is an open subset. Furthermore, it is possible
to achieve that $V \subset \mathbb{R}^n$ is convex. With such adapted
local charts $(\tilde{U}, \tilde{x})$ of $P$ and $(U, x)$ of $M$,
which we will use throughout, we have the commutative diagram
\begin{equation}
    \label{eq:diagram-surjective-submersion}
    \xymatrix{P \ar[d]_{\mathsf{p}}&& \mathsf{p}^{-1}(U) \ar
      @{_{(}->}[ll]_{\iota} 
      \ar[d]_{\mathsf{p}} && \tilde{U} \ar @{_{(}->}[ll]_{\iota}
      \ar[d]_{\mathsf{p}} \ar[rr]^{\tilde{x}}&& **[r] V\times G\subset
      \mathbb{R}^{n+k} 
      \ar[d]_{\mathsf{pr}_1}\\
      M && U \ar @{_{(}->}[ll]_{\iota}\ \ar @{=}[rr] && U \ar[rr]^x &&
      **[r] V\subset \mathbb{R}^n,
    }
\end{equation}
where $\iota$ in every case denotes the embedding map. Every column
of this diagram is again the diagram of a surjective submersion with a
canonical right module structure for the corresponding algebras of
smooth functions.
\begin{lemma}
    \label{lemma:localization}
    The restriction maps and local charts give rise to chain maps
    \begin{equation}
        \label{eq:localization}
        \xymatrix{\mathrm{HC}^\bullet_{\mathrm{diff}} (M,\Diffop(P))
          \ar[d] & \\ 
        \mathrm{HC}^\bullet_{\mathrm{diff}}
        (U,\Diffop(\mathsf{p}^{-1}(U))) 
        \ar[d] & \\
        \mathrm{HC}^\bullet_{\mathrm{diff}} (U,\Diffop(\tilde{U}))
        \ar[r]^(.45){\cong} &
        \mathrm{HC}^\bullet_{\mathrm{diff}} (V,\Diffop(V\times G)).
        }
    \end{equation}
\end{lemma}

In the second step we now compute the Hochschild cohomology
$\mathrm{HH}^\bullet_{\mathrm{diff}} (V, \Diffop(V\times G))$.
\begin{lemma} 
    Proposition~\ref{proposition:isomorphism-Hochschild-bar} can be
    applied for $\Diffop (V \times G)$. Hence we have
    \begin{equation}
        \label{eq:isomorphism-Hochschild-bar-surj-sub}
        \mathrm{HC}^\bullet_{\mathrm{diff}} (V, \Diffop(V\times G))
        \cong 
        \Union_{l=0}^\infty \Hom_{\mathcal{A}^e}^{\mathrm{diff}}
        (X_\bullet, \Diffop^l(V\times G))
    \end{equation}
    and
    \begin{equation}
        \label{eq:equivalence-Hochschild-Koszul-cohomology}
        \mathrm{HH}^\bullet_{\mathrm{diff}} (V,\Diffop(V\times G))
        \cong  
        \mathrm{H} \left(
            \Union_{l=0}^\infty \Hom_{\mathcal{A}^e}
            (K_\bullet,\Diffop^l(V\times G))
        \right).
    \end{equation}
\end{lemma}
\begin{proof}
    For $\Diffop(V\times G)=\bigcup_{l=0}^\infty \Diffop^l(V\times G)$
    the preconditions of Proposition
    \ref{proposition:isomorphism-Hochschild-bar} are easily checked
    using the natural Fr{\'e}chet topology and the given
    $(\mathcal{A}, \mathcal{A})$-bimodule structure
    \eqref{eq:FunnyBimoduleStructure} of $\Diffop^l (V\times G)$.
    Condition~\eqref{eq:right-module-structure-differential} is a
    simple consequence of the Leibniz rule. By
    Lemma~\ref{lemma:local-form-of-Diffops} and
    Remark~\ref{remark:local-form-of-Diffops} the algebraic notion of
    a differential operator and the one of
    Definition~\ref{definition:differential-maps} coincide. 
\end{proof}
Note that in particular the notions of $\HCdiff$ as in
Example~\ref{example:DifferentialDeformations} and
Section~\ref{sec:homological} coincide in this case justifying thereby
our previous abuse of notation.

The crucial step for all further considerations is the following
explicit computation of the cohomology of $\bigcup_{l=0}^\infty
\Hom_{\mathcal{A}^e} (K_\bullet, \Diffop^l (V \times G))$.
\begin{remark}
    \label{remark:symbol-calculus}
    In the following we make use of the well-known symbol calculus for
    differential operators. Depending on a torsion-free covariant
    derivative $\nabla$ on a manifold $M$, every differential operator
    $D \in \Diffop^l(M)$ can be identified with a unique series
    $T_0, \ldots, T_l$ of symmetric multi-vector fields $T_j \in
    \Gamma^\infty(S^j TM)$, $j= 0,\ldots ,l$, which yield
    $D=\sum_{j=0}^l D_{T_j}$. Here,
    \begin{equation}
        \label{eq:symbol-calculus}
        D_{T_j}(a)=\frac{1}{j!} 
        \left\langle T_j, \mathsf{D}^{(j)} a \right\rangle
        \quad \forall a \in C^\infty(M),
    \end{equation}
    where $\langle \cdot, \cdot \rangle$ denotes the natural pairing
    of symmetric multi-vector fields with symmetric differential forms
    and $\mathsf{D}^{(j)}= \frac{1}{j!}\mathsf{D}^j$ denotes the
    $j$-fold symmetrized covariant derivative, which in local
    coordinates is given by $\mathsf{D}= \mathrm{d} x^i \vee
    \nabla_{\frac {\partial} {\partial x^i}}$. Finally,
    $\sigma(D)=T_0+ \cdots + T_l\in \bigoplus_{r=0}^\infty
    \Gamma^\infty(S^rTM)$ is called the symbol of $D$ with respect to
    $\nabla$.
\end{remark}

In the present situation we can use the product structure of $V \times
G$ and the fact that $V \subset \mathbb{R}^n$ to find that every
operator $D \in \Diffop^l (V\times G)$ can be written in the form
\begin{equation}
    \label{eq:diffop-on-product-manifold}
    Df = 
    \sum_{r = 0}^l
    \underbrace{
      \sum_{s = r}^l T^{i_1\ldots i_r}_{j_1\ldots j_{s-r}}
      \left\langle 
          X^{j_1} \vee \cdots \vee X^{j_{s-r}}, 
          \mathsf{D}_G^{(s-r)} 
          \circ
          \frac{\partial^r f}{\partial x^{i_1} \cdots \partial x^{i_r}}
      \right\rangle
    }_{= D_r},
\end{equation}
where $T^{i_1 \ldots i_r}_{j_1 \ldots j_{s-r}} \in C^\infty (V \times
G)$ and $X^j \in \Gamma^\infty (TG) \subset \Gamma^\infty (T (V \times
G))$. Here, the symmetrized covariant derivative $\mathsf{D}_G$
belongs to a torsion-free covariant derivative $\nabla^G$ on $G$,
which both can be extended to $V \times G$.  Thus we have a
decomposition
\begin{equation}
    \label{eq:decomposition-diffops-on-product}
    \Diffop^l(V\times G) = \bigoplus_{r=0}^l \Diffop^l_r(V\times G)
\end{equation}
where $D_r\in \Diffop^l_r(V\times G)$ has the form as in
\eqref{eq:diffop-on-product-manifold}. By linear extension of
\begin{equation}
    \label{eq:degree}
    \deg D_r= r D_r \quad \forall D_r\in \Diffop^l_r(V\times G)
\end{equation}
we obtain a map $\deg: \Diffop^\bullet (V\times G)\longrightarrow
\Diffop^\bullet(V\times G)$. With the decomposition
\eqref{eq:decomposition-diffops-on-product} we get
\begin{equation}
    \label{eq:decomposition-Hom-Koszul-complex}
    \Hom_{\mathcal{A}^e} (K_\bullet,\Diffop^l(V\times G))=
    \bigoplus_{r=0}^l \Hom_{\mathcal{A}^e}
    (K_\bullet,\Diffop^l_r(V\times G)) 
\end{equation}
for all $l \in \mathbb{N}$. Thus we can consider the corresponding map
$\deg: \bigcup_{l=0}^\infty \Hom_{\mathcal{A}^e} (K_\bullet, \Diffop^l
(V \times G)) \longrightarrow \bigcup_{l=0}^\infty
\Hom_{\mathcal{A}^e} (K_\bullet, \Diffop^l (V\times G))$ which is
again defined by linear extension of
\begin{equation}
    \label{eq:degree-Hom}
    \deg \psi =r \psi \quad \forall \psi \in
    \Hom_{\mathcal{A}^e}(K_\bullet,\Diffop^l_r(V\times G)).
\end{equation}
Clearly, $(\deg \psi)(\omega)= \deg (\psi(\omega))$ for all $\psi \in
\Hom_{\mathcal{A}^e} (K_k, \Diffop^l (V\times G))$, and $\omega \in
K_k$ with $l, k \in \mathbb{N}_0$. The following lemma is obvious.
\begin{lemma}
    \label{lemma:coordinate-functions}
    Let $D \in \Diffop (V\times G)$ and $x^i: V \longrightarrow
    \mathbb{R}$ be the coordinate functions with respect to the
    canonical basis $\{e_i\}_{i = 1, \ldots, n}$ of $\mathbb{R}^n$.
    Moreover, define $\xi^i = x^i \otimes 1 - 1 \otimes x^i \in
    C^\infty(V \times V)$. Then
    \begin{equation}
        \label{eq:degD}
        \sum_{j=1}^n 
        \left(
            (D \cdot x^j) \circ \frac{\partial}{\partial x^j} 
            - (x^j \cdot D) \circ \frac{\partial}{\partial x^j}
        \right) 
        =
        \deg D
    \end{equation}
    and
    \begin{equation}
        \sum_{j=1}^n 
        \left( 
            \left(
                (-\xi^j) \cdot D
            \right)
            \circ \frac{\partial}{\partial x^j}
        \right)
        =
        \deg D.
    \end{equation}
\end{lemma}
\begin{definition}
    \label{definition:homotopy-maps-Koszul}
    We define the map
    \begin{equation}
        \label{eq:homotopy-map}
        \delta_K^{-1}:
        \Union_{l=0}^\infty \Hom_{\mathcal{A}^e} 
        (K_\bullet, \Diffop^l (V \times G))
        \longrightarrow 
        \Union_{l=0}^\infty \Hom_{\mathcal{A}^e}
        (K_{\bullet-1}, \Diffop^l (V \times G)) 
    \end{equation}
    by the linear extension of the maps
    \begin{equation}
        \label{eq:homotopy-map-explicitly}
        (\delta_K^{-1})^k \psi
        =
        \begin{cases}
            \frac{1} {k+r} (\delta_K^*)^k\psi 
            &
            \textrm{for} \; k+r \neq 0 \\
            0 
            &
            \textrm{for} \;  k=0=r
        \end{cases}
    \end{equation}    
    for $\psi \in \Hom_{\mathcal{A}^e}(K_k,\Diffop^l_r(V\times G))$
    and $k\ge 0$, where the $\mathcal{A}^e$-linear map
    \begin{equation}
        \label{eq:deltastarDef}
        (\delta_K^*)^k: 
        \Union_{l=0}^\infty \Hom_{\mathcal{A}^e} 
        (K_k, \Diffop^l (V \times G)) 
        \longrightarrow 
        \Union_{l=0}^\infty \Hom_{\mathcal{A}^e}
        (K_{k-1}, \Diffop^l (V \times G))
    \end{equation}
    is defined by
    \begin{equation}
        \label{eq:homotopy-map-Hom-Koszul}
        ((\delta_K^*)^k\psi) 
        (e^{i_1} \wedge \cdots \wedge e^{i_{k-1}})
        = 
        - \sum_{j=1}^n
        \psi (e^j \wedge e^{i_1} \wedge \cdots \wedge e^{i_{k-1}})
        \circ \frac{\partial}{\partial x^j}.
    \end{equation}
\end{definition}
\begin{remark}
    \label{remark:Koszul-homotopy-diff}
    Note that by definition we have
    \begin{equation}
        \label{eq:homotopy-degrees}
        \delta_K^{-1}:
        \Hom_{\mathcal{A}^e}(K_\bullet,\Diffop^l_r(V\times G)) 
        \longrightarrow
        \Hom_{\mathcal{A}^e}(K_{\bullet-1},\Diffop^{l+1}_{r+1}(V\times
        G))
        \quad \forall r\le l\in \mathbb{N}_0.
    \end{equation}
\end{remark}
These maps turn out to be explicit homotopies for the Koszul complex:
\begin{proposition}
    For $1 \le k \in \mathbb{N}$ the map $\delta_K^{-1}$ yields an
    explicit homotopy for the Koszul complex $\bigcup_{l=0}^\infty
    \Hom_{\mathcal{A}^e} (K_\bullet,\Diffop^l(V\times G))$, i.e.
    \begin{equation}
        \label{eq:homotopy}
        \delta_K^{k-1}\circ (\delta_K^{-1})^k +
        (\delta_K^{-1})^{k+1}\circ \delta_K^k  =
        \id.
    \end{equation}
    It follows that the Koszul cohomology is trivial for $k \ge 1$, i.e.
    \begin{equation}
        \label{eq:Koszul-cohomology-vanishes}
        \mathrm{H} \left(
            \Union_{l=0}^\infty
            \Hom_{\mathcal{A}^e}(K_k, \Diffop^l (V \times G))
        \right)
        = \{0\}.
    \end{equation}
\end{proposition}
\begin{proof}
    With use of the maps $\xi^i$ of
    Lemma~\ref{lemma:coordinate-functions} it is easy to see that the
    differentials $\delta_K^k$ can be written as
    \begin{equation}
        \label{eq:differential-Hom-Koszul}
        (\delta_K^k\psi)
        (e^{i_1} \wedge \cdots \wedge e^{i_{k+1}})
        =
        \psi \left(
            \sum_{j=1}^n \xi^j \ins_{\mathrm{a}}(e_j) 
            (e^{i_1}\wedge \cdots \wedge e^{i_{k+1}}) 
        \right) \quad \forall i_s \in \{1,\ldots, n\},
    \end{equation}
    where $\ins_{\mathrm{a}}$ denotes the insertion map for the
    antisymmetric forms of $\mathbb{R}^n$.
    Lemma~\ref{lemma:coordinate-functions} leads to
    \begin{equation}
        \label{eq:homotopy-deg}
        \delta_K^{k-1}\circ (\delta_K^*)^k + (\delta_K^*)^{k+1}\circ
        \delta_K^k  = \deg + k \id,
    \end{equation}
    which finally yields \eqref{eq:homotopy}. The last statement is an
    immediate consequence.
\end{proof}
\begin{theorem}[Hochschild cohomology, local situation]
    \label{theorem:local-Hochschild-cohomologies-surjective-submersions}
    Let $V \subset \mathbb{R}^n$ be an open and convex subset and let
    $G$ be a manifold. Then we have for all $1 \le k \in \mathbb{N}$: 
    \begin{enumerate}
    \item The $k$-th differential Hochschild cohomology with values in
        $\Diffop(V \times G)$ is trivial
        \begin{equation}
            \label{eq:cohomology-groups-local}
            \mathrm{HH}^k_{\mathrm{diff}} (V,\Diffop(V\times G))= \{0\}.
        \end{equation}
    \item There exists an explicit homotopy
        \begin{equation}
            \label{eq:explicit-homotopy}
            \delta^{k-1} \circ (\delta^{-1})^k+ (\delta^{-1})^{k+1}
            \circ \delta^k=\id,
        \end{equation}
        where the maps $(\delta^{-1})^k: \HCdiff^k (V, \Diffop(V
        \times G)) \longrightarrow \HCdiff^{k-1} (V, \Diffop (V \times
        G))$ are given by
        \begin{equation}
            \label{eq:explicit-homotopy-map}
            (\delta^{-1})^k
            = \Xi \circ 
            \left(
                (G^{k-1})^* \circ (\delta_K^{-1})^{k} \circ (F^k)^* 
                + (s^{k-1})^*
            \right) \circ \Xi^{-1}.
        \end{equation}
    \item In particular, for any multi-index $L = (l_1, \ldots,
        l_k) \in \mathbb{N}_0^k$ we have
        \begin{equation}
            \label{eq:explicit-homotopy-map-diff}
            \delta^{-1}:
            \Diffop^L (V, \Diffop^l (V \times G))
            \longrightarrow \Diffop^{\tilde{L}} (V, \Diffop^{l+1} (V\times G)),
        \end{equation} 
        where the new multi-index $\tilde{L} = (\tilde{l}_1, \ldots,
        \tilde{l}_{k-1}) \in \mathbb{N}_0^{k-1}$ is given by
        \begin{equation}
            \label{eq:new-multi-index}
            \tilde{l}_i = \max \{l_i, l+2\} \quad \forall i = 1, \ldots, k-1.
        \end{equation}
    \end{enumerate}
\end{theorem}
\begin{proof}
    The proof of equation \eqref{eq:explicit-homotopy} is a simple
    computation which makes use of \eqref{eq:homotopy},
    \eqref{eq:id-Theta-homotopic} and the properties of the involved
    functions. Then, equation \eqref{eq:cohomology-groups-local} is
    trivial. With \eqref{eq:map-for-id-Theta-homotopic} we have
    $(s^{k-1})^*= (h_X^{k-1})^*\circ (\id - (G^k)^*\circ (F^k)^*)$.
    Thus, the third assertion \eqref{eq:explicit-homotopy-map-diff} is
    clear with Proposition \ref{proposition:restrictions-diff} and
    Remark \ref{remark:Koszul-homotopy-diff} by counting the orders of
    differentiation, since $\Xi$ does not change any of them.
\end{proof}

In the third and last step we use this local result to compute the
Hochschild cohomologies for arbitrary surjective submersions. For this
purpose we consider the vertical differential operators $D \in
\Diffopver(P)$ which are defined by the condition
\begin{equation}
    \label{eq:vertical-diffops}
    D(f\cdot \mathsf{p}^*a)=
    D(f)\cdot \mathsf{p}^*a \quad \forall a\in
    C^\infty(M), f\in C^\infty(P).
\end{equation}
\begin{theorem}[Hochschild cohomology for surjective submersions]
    \label{theorem:Hochschild-cohomologies-surjective-submersions}
    Let $\mathsf{p}: P\longrightarrow M$ be a surjective submersion.
    Then   
    \begin{equation}
        \label{eq:cohomology-groups-surjective-submersion}
        \mathrm{HH}^k_{\mathrm{diff}} (M,\Diffop(P))
        =
        \begin{cases}
            \Diffopver(P) & \textrm{for} \; k = 0 \\
            \{0\} & \textrm{for} \; k \ge 1.
        \end{cases}
    \end{equation}
    More specifically, every $\phi \in \Diffop^L(M, \Diffop^l(P))$
    with $L = (l_1, \ldots, l_k) \in \mathbb{N}_0^k$, $k \ge 1$, and
    $\delta \phi = 0$ is of the form
    \begin{equation}
        \label{eq:phiExact}
        \phi = \delta \Theta,
    \end{equation}
    where $\Theta \in \Diffop^{\tilde{L}}(M, \Diffop^{l+1}(P))$ and
    $\tilde{L}$ as in \eqref{eq:new-multi-index}.
\end{theorem}
\begin{proof}
    The case $k = 0$ is clear from the definition
    \eqref{eq:vertical-diffops}. For $k\ge 1$ we consider atlases
    $\{(\tilde{U}_\alpha ,\tilde{x}_\alpha)\}_{\alpha \in I}$ of $P$
    and $\{(U_\alpha ,x_\alpha)\}_{\alpha \in I}$ of $M$ consisting of
    adapted local charts. In addition, let
    $\{\tilde{\chi}_\alpha\}_{\alpha \in I}$ be a locally finite
    partition of unity for $P$ which is subordinate to the open cover
    $\{\tilde{U}_\alpha\}_{\alpha \in I}$, so $\tilde{\chi}_\alpha \in
    C^\infty(P)$ with $\supp \tilde{\chi}_\alpha \subset
    \tilde{U}_\alpha$ and $\sum_{\alpha \in I} \tilde{\chi}_\alpha =
    1$.  Then, let $\phi \in \mathrm{HC}^k_{\mathrm{diff}} (M,
    \Diffop(P))$ be closed, $\delta \phi =0$. By definition there
    exists an $l \in \mathbb{N}_0$ and a multi-index $L \in
    \mathbb{N}_0^k$ with $\phi \in \Diffop^L(M, \Diffop^l(P))$.  Due
    to Lemma~\ref{lemma:localization} the restrictions
    $\phi_{\tilde{U}_\alpha} \in \Diffop^L (U_\alpha, \Diffop^l
    (\tilde{U}_\alpha))$ are closed, too.
    Theorem~\ref{theorem:local-Hochschild-cohomologies-surjective-submersions}
    then ensures that these elements are exact: with
    \eqref{eq:explicit-homotopy} there exist $\Theta_\alpha \in
    \Diffop^{\tilde{L}} (U_\alpha,\Diffop^{l+1} (\tilde{U}_\alpha))$,
    $\tilde{L} \in \mathbb{N}_0^{k-1}$ as in
    \eqref{eq:new-multi-index}, with $\delta \Theta_\alpha=
    \phi_{\tilde{U}_\alpha}$. The restrictions
    \[
    \Theta_{\tilde{\chi}_\alpha} 
    (a_1, \ldots, a_{k-1})|_{\tilde{U}_\alpha}
    = 
    \tilde{\chi}_\alpha|_{\tilde{U}_\alpha} 
    \Theta_\alpha 
    \left(
        a_1|_{U_\alpha}, \ldots, a_{k-1}|_{U_\alpha}
    \right)
    \]
    and $\Theta_{\tilde{\chi}_\alpha}(a_1, \ldots, a_{k-1})|_{P
      \setminus \supp \tilde{\chi}_\alpha} = 0$ define global elements
    $\Theta_{\tilde{\chi}_\alpha} \in
    \Diffop^{\tilde{L}}(M,\Diffop^{l+1}(P))$. It is an easy
    computation which shows that $\delta \Theta_{\tilde{\chi}_\alpha}=
    \tilde{\chi}_\alpha \phi$. Due to the local finiteness of the
    partition of unity it is clear that $\Theta = \sum_{\alpha \in
      I}\Theta_{\chi_\alpha} \in
    \mathrm{HC}^{k-1}_{\mathrm{diff}}(M,\Diffop^{l+1} (P))$ is a
    well-defined differential operator, which finally yields the
    aspired exactness of $\phi$ via $\Theta$ fulfilling the
    requirements in \eqref{eq:phiExact},
    \[
    \phi
    = \left(\sum_{\alpha\in I} \tilde{\chi}_\alpha \right) \phi
    = \sum_{\alpha\in I}(\tilde{\chi}_\alpha \phi)
    = \sum_{\alpha\in I} \delta \Theta_{\tilde{\chi}_\alpha} 
    = \delta \sum_{\alpha\in I} \Theta_{\tilde{\chi}_\alpha} 
    = \delta \Theta. 
    \]
\end{proof}

Due to this result and those of
Section~\ref{sec:AlgebraicPreliminaries} we find the existence and the
uniqueness of deformed right module structures. So, the first part of
Theorem~\ref{theorem:def-quant-surj-sub} is proven.
\begin{corollary}
    Every surjective submersion admits a deformation quantization
    which is unique up to equivalence.
\end{corollary}

In order to complete the proof of
Theorem~\ref{theorem:def-quant-surj-sub}, we have to show that there
always exists a deformation quantization which preserves the
fibration. For this purpose we have to make some choices of
geometrical structures: first we choose a connection on $P$, i.e.\  a
decomposition
\begin{equation}
    \label{eq:connection}
    TP = VP \oplus HP
\end{equation}
of the tangent bundle into the canonically given vertical bundle $VP =
\ker T\mathsf{p}$ and a horizontal bundle $HP$. Then, any vector field
$X\in \Gamma^\infty(TM)$ has a horizontal lift $X^{\mathrm{h}}\in
\Gamma^\infty(HP)$ which is uniquely defined by the two demands that
$X^{\mathrm{h}}$ is horizontal and $\mathsf{p}$-related to $X$,
$T\mathsf{p}\circ X^{\mathrm{h}}= X\circ \mathsf{p}$. Second, we
choose an always existing torsion-free covariant derivative $\nabla^P$
on $TP$, which respects the vertical bundle. This means that
$\nabla^P_Z V \in \Gamma^\infty(VP)$ for all vertical vector fields $V
\in \Gamma^\infty(VP)$ and arbitrary vector fields $Z \in
\Gamma^\infty(TP)$. The following lemma is a simple computation.
\begin{lemma}
    \label{lemma:connection-covariant-derivatives}
    Let $TP = VP\oplus HP$ be a connection and let $\nabla^P$ be a
    torsion-free covariant derivative, which respects the vertical
    bundle. Then we have:
    \begin{enumerate}
    \item The equation        
        \begin{equation}
            \label{eq:covariant-derivative-basis}
            T\mathsf{p}\circ \nabla^P_{X^{\mathrm{h}}}Y^{\mathrm{h}} =
            \left(\nabla^M_X Y \right)\circ \mathsf{p} \quad \forall
            X,Y \in \Gamma^\infty(TM)
        \end{equation}
        defines a torsion-free covariant derivative $\nabla^M$ on $M$.
    \item For all $l \in \mathbb{N}$, the corresponding symmetrized
        covariant derivatives $\mathsf{D}_P^{(l)}$ and
        $\mathsf{D}_M^{(l)}$ are related by
        \begin{equation}
            \label{eq:symmetric-covariant-derivatives-relation}
            \mathsf{D}_P^{(l)} \circ \mathsf{p}^*=
            \mathsf{p}^*\circ \mathsf{D}_M^{(l)}.
        \end{equation}
    \end{enumerate}
\end{lemma}
\begin{lemma}
    \label{lemma:special-diffop}
    Let $\bullet$ be a deformation quantization of a surjective
    submersion $\mathsf{p}: P \longrightarrow M$. Then there exists a
    formal series $T = \id + \sum_{r=1}^\infty \lambda^r
    T_r$ of differential operators $T_r \in
    \Diffop(P)$ such that
    \begin{equation}
        \label{eq:Diffop-deformation-quantization}
        T(\mathsf{p}^*a)
        = 1 \bullet a \quad \forall a\in C^\infty(M).
    \end{equation}
\end{lemma}
\begin{proof}
    With $f\bullet a =\sum_{r=0}^\infty \lambda^r\rho_r(f,a)$ it is
    obvious that $D_r(a)=\rho_r(1,a)$ for all $r$ defines a
    differential operator $D_r\in \Diffop^{l_r}
    (M,C^\infty(P))$ with $l_r\in \mathbb{N}_0$. An
    according symbol calculus with respect to the structures of Lemma
    \ref{lemma:connection-covariant-derivatives} shows that
    \begin{equation}
        \label{eq:diff-M-P}
        D_r(a) = \sum_{s=0}^{l_r} 
        \left\langle 
            T^r_s, \mathsf{p}^* \mathsf{D}_M^{(s)} a
        \right\rangle 
        \quad \textrm{with} \quad T^r_s \in \Gamma^\infty(HP).
    \end{equation}
    Then we define $T_r \in \Diffop^{l_r}(P)$ by
    \begin{equation}
        \label{eq:diff-P}
        T_r(f)
        =  \sum_{s=0}^{l_r} 
        \left\langle T^r_s, \mathsf{D}_P^{(s)} f \right\rangle. 
    \end{equation}
    Clearly, $T_0=\id$ and with
    \eqref{eq:symmetric-covariant-derivatives-relation} we find
    $T_r(\mathsf{p}^* a)= D_r(a)$ which proves the lemma.
\end{proof}

With this lemma we now conclude the proof of
Theorem~\ref{theorem:def-quant-surj-sub}.
\begin{corollary}
    Every surjective submersion admits a deformation quantization
    which preserves the fibration.
\end{corollary}
\begin{proof}
    Let $\bullet$ be an arbitrary deformation quantization. Since the
    map $T$ in Lemma~\ref{lemma:special-diffop} has all properties of
    an equivalence transformation, $f\bullett a= T^{-1}\left(Tf
        \bullet a \right)$ defines a new deformation quantization
    $\bullett$. With \eqref{eq:Diffop-deformation-quantization} we
    then get $(\mathsf{p}^*a) \bullett b = \mathsf{p}^*( a\star b)$
    for all $a,b \in C^\infty(M)$.
\end{proof}

Theorem \ref{theorem:Hochschild-cohomologies-surjective-submersions}
further ensures that we can apply
Proposition~\ref{proposition:Kommutante}. Thus, every choice of a
deformation quantization $\bullet$ and a decomposition
\begin{equation}
    \label{eq:decomposition-diffop}
    \Diffop(P)=\Diffopver(P)\oplus
    \overline{\Diffopver(P)}
\end{equation}
leads to an isomorphism
\begin{equation}
    \label{eq:iso-commutant-surj-sub}
    \rho': 
    \Diffopver(P) [[\lambda]]
    \longrightarrow 
    \{ D\in \Diffop(P) [[\lambda]] \: | \: D(f\bullet a) = D(f)\bullet a\}
\end{equation}
with $\rho'=\id + \sum_{r=1}^\infty \lambda^r \rho'_r$ and an
associative deformation $\left(\Diffopver (P) [[\lambda]], \mu'
\right)$ of $\Diffopver(P)$. Using the definitions
\begin{equation}
    \label{eq:bulletp-starp}
    A \bulletp f = \rho'(A)f
    \quad 
    \textrm{and}
    \quad A \starp B= \mu'(A,B),
\end{equation}
for all $A, B \in \Diffopver(P) [[\lambda]]$ and $f \in
C^\infty(P)[[\lambda]]$, the functions $C^\infty(P)[[\lambda]]$
inherit a $(\starp, \star)$-bimodule structure which is shortly
denoted by
\begin{equation}
    \label{eq:bimodule-surjective-submersion}
    {}_{(\Diffopver (P)[[\lambda]],\starp)} 
    (\bulletp,C^\infty(P)[[\lambda]], \bullet)_{(C^\infty(M)[[\lambda]],\star)}.
\end{equation}
\begin{proposition}
    \label{proposition:MutualCommutants}
    The commutant of $\Diffopver(P)[[\lambda]]$ acting via $\bulletp$
    on $C^\infty(P)[[\lambda]]$ is given by $C^\infty(M)[[\lambda]]$
    acting via $\bullet$. Thus the two algebras in
    \eqref{eq:bimodule-surjective-submersion} are mutual commutants.
\end{proposition}
\begin{proof}
    Let $A = \sum_{r=0}^\infty \lambda^r A_r$ satisfy $A(D \bulletp f)
    = D \bulletp (Af)$ for all $D \in \Diffopver(P)[[\lambda]]$ and $f
    \in C^\infty(P)[[\lambda]]$. It follows that in zeroth order $A_0$
    commutes with all undeformed left multiplications with functions
    $D \in C^\infty(P)$. Hence $A_0 \in C^\infty(P)$. Moreover, since
    $A_0$ commutes with all $D \in \Diffopver(P)$ in zeroth order,
    $A_0$ is constant in fibre directions, i.e.\  $A_0 =
    \mathsf{p}^*a_0$ for some $a_0 \in C^\infty(M)$.  Since $A(D
    \bulletp f) - (D \bulletp f) \bullet a_0 = D \bulletp (Af - f
    \bullet a_0)$ and since $Af - f \bullet a_0$ has vanishing zeroth
    order we can proceed by induction.
\end{proof}
\begin{remark}
    \label{remark:NotYetMorita}
    Even though the two algebras are mutual commutants the bimodule is
    \emph{not} a Morita equivalence bimodule. This is not even true on
    the classical level. However, we will see in
    Section~\ref{sec:AssociatedBundles} the relation to Morita theory.
\end{remark}

%
%

\section{Deformation Quantization of Principal Bundles}
\label{sec:PrincipalBundles}

In this section we will prove
Theorem~\ref{theorem:def-quant-principal} and therefore consider
deformations of the right module structures which appear in the
special case of principal fibre bundles. As we will see, the general
results for surjective submersions of
Section~\ref{sec:SurjectiveSubmersions} can be applied and even
simplified in some parts.

So, in this section, let $\mathsf{p}: P\longrightarrow M$ be a
principal fibre bundle with total space $P$, basis $M$, structure Lie
group $G$ and principal right action $\mathsf{r}: P \times G
\longrightarrow P$ as in the introduction.  As usual we write
$\mathsf{r}_g (u) = \mathsf{r}(u, g)$ for $u \in P$ and $g \in G$.
Now, the undeformed right module structure
\eqref{eq:right-module-structure-surjective-submersion} has the
further property of $G$-invariance,
\begin{equation}
    \label{eq:G-invariance-right-module-structure}
    \mathsf{r}_g^* \circ \rho_0(a)= \rho_0(a)\circ \mathsf{r}_g^*.
\end{equation}
Since the aspired deformation should preserve this, we now have to consider the
$G$-invariant differential operators
\begin{equation}
    \label{eq:G-invariant-differential-operators}
    \Diffop(P)^G= \{ D\in
    \Diffop(P) \: | \: 
    \mathsf{r}_g^* \circ D = D \circ \mathsf{r}_g^* 
    \quad \forall g \in G\}
\end{equation}
instead of all differential operators.  It is clear that the bimodule
structure \eqref{eq:FunnyBimoduleStructure} is compatible with the
right action, i.e.\  $\Diffop(P)^G$ is a sub-bimodule. Following our
general framework we have to use the differential Hochschild complex
with values in this bimodule $\Diffop(P)^G$.

Due to the properties of a principal fibre bundle, for any $p\in M$
there exist an open subset $U\subset M$ with $p\in U$ and
diffeomorphisms $x: M \supseteq U \longrightarrow V \subseteq
\mathbb{R}^n$ and $\varphi: P \supseteq \mathsf{p}^{-1} (U)
\longrightarrow U\times G$ with $\mathsf{pr}_1 \circ \varphi
=\mathsf{p}$ and $\varphi \circ \mathsf{r}_g = (\id_U \times r_g)
\circ \varphi$. Here, $r_g (h) = hg$ is the right multiplication with
$g \in G$ on $G$. As before, $(U, x)$ is a chart of $M$ and
$(\mathsf{p}^{-1} (U), \varphi)$ is a fibre bundle chart of $P$,
respectively. Again, it is possible to achieve that $V\subseteq
\mathbb{R}^n$ is convex. Altogether, this leads to the commutative
diagram
\begin{equation}
    \xymatrix{
      P \ar@(ur,ul)[] _{\mathsf{r}_g} \ar[d]_{\mathsf{p}}&& 
      \mathsf{p}^{-1}(U)
      \ar@(ur,ul)[] _{\mathsf{r}_g} \ar @{_{(}->}[ll]_{\iota}
      \ar[d]_{\mathsf{p}} \ar[rr]^{\varphi}&& 
      U\times G \ar@(ur,ul)[]
      _{\id_U \times r_g}
      \ar[d]_{\mathsf{pr}_1} \ar[rr]^{x\times \id_G}&&
      V\times G
      \ar@(ur,ul)[] _{\id_V \times r_g}
      \ar[d]_{\mathsf{pr}_1}\\
      M && 
      U \ar @{_{(}->}[ll]_{\iota}\ \ar @{=}[rr] &&
      U \ar[rr]^x &&
      V\subseteq \mathbb{R}^n.
    }
\end{equation}
Again, every column of this diagram describes a principal fibre bundle
with the same structure group $G$. Similar to Lemma
\ref{lemma:localization} the following statement holds.
\begin{lemma}
    \label{lemma:localization-principal}
    The restriction map, the fibre bundle
    chart and the chart of $M$ give rise to chain maps
        \begin{equation}
        \label{eq:localization-principal}
        \xymatrix{
          \mathrm{HC}^\bullet_{\mathrm{diff}} (M,\Diffop(P)^G)
          \ar[d] &\\ 
          \mathrm{HC}^\bullet_{\mathrm{diff}}
          (U,\Diffop(\mathsf{p}^{-1}(U))^G) 
          \ar[r]^(.5){\cong} &
          \mathrm{HC}^\bullet_{\mathrm{diff}} (V,\Diffop(V\times G)^G).
        }
    \end{equation}
\end{lemma}

The computation of the Hochschild cohomologies
$\mathrm{HH}^\bullet_\mathrm{diff} (V, \Diffop (V \times G)^G)$, now
with the restriction to the $G$-invariant differential operators with
respect to the action $\id_V \times r_g$ of the trivial principal fibre
bundle, can be done in an absolutely analogous way to that for
surjective submersions yielding the same result: in degrees $k \ge 1$
the cohomology is trivial. The reason for this are the following
facts:
\begin{enumerate}
\item $\Diffop^l(V\times G)^G\subset \Diffop^l(V\times G)$ is a closed
    subspace for all $l \in \mathbb{N}_0$ and hence a topological
    Hausdorff and complete bimodule itself.
\item The differential operators $\frac{\partial}{\partial x^j} \in
    \Diffop^1(V\times G)^G$ for $j = 1, \ldots, n$ in
    \eqref{eq:homotopy-map-Hom-Koszul} are $G$-invariant. Thus,
    \begin{equation}
        \label{eq:homotopy-map-principal}
        \delta_K^{-1}:
        \Union_{l=0}^\infty \Hom_{\mathcal{A}^e}
        (K_\bullet,\Diffop^l(V\times G)^G) 
        \longrightarrow
        \Union_{l=0}^\infty \Hom_{\mathcal{A}^e}
        (K_{\bullet-1},\Diffop^l(V\times G)^G).
    \end{equation}
\end{enumerate}
Altogether, we find the following theorem about the Hochschild
cohomologies for principal fibre bundles.
\begin{theorem}[Hochschild cohomology for principal fibre bundles]
    \label{theorem:Hochschild-cohomologies-principal-fibre-bundles}
    Let $\mathsf{p}: P\longrightarrow M$ be a principal fibre bundle
    with structure Lie group $G$.
    Then
    \begin{equation}
        \label{eq:cohomology-groups-principal-fibre-bundle}
        \mathrm{HH}^k_{\mathrm{diff}} (M,\Diffop(P)^G)
        =
        \begin{cases}
            \Diffopver(P)^G & \textrm{for} \; k = 0 \\
            \{0\} & \textrm{for} \; k\ge 1.
        \end{cases}
    \end{equation}
    Again, every $\phi \in \Diffop^L(M, \Diffop^l(P)^G)$ with $L =
    (l_1, \ldots, l_k) \in \mathbb{N}_0^k$, $k \ge 1$, and $\delta
    \phi = 0$ is of the form
    \begin{equation}
        \label{eq:phiExactInvariant}
        \phi = \delta \Theta,
    \end{equation}
    where $\Theta \in \Diffop^{\tilde{L}}(M, \Diffop^{l+1}(P)^G)$ and
    $\tilde{L}$ as in \eqref{eq:new-multi-index}.
\end{theorem}
\begin{proof}
    The proof for $k = 0$ is again trivial. For $k \ge 1$ we choose an
    atlas $\{(U_\alpha, x_\alpha)\}_{\alpha\in I}$ of $M$ and an
    appropriate principal fibre bundle atlas $\{(\mathsf{p}^{-1}
    (U_\alpha), \varphi_\alpha)\}_{\alpha\in I}$ of $P$. Further, let
    $\{\chi_\alpha\}_{\alpha \in I}$ be a partition of unity for $M$
    which is subordinate to the open cover $\{U_\alpha\}_{\alpha \in
      I}$, so $\chi_\alpha \in C^\infty(M)$ with $\supp \chi_\alpha
    \subseteq U_\alpha$ and $\sum_{\alpha \in I} \chi_\alpha = 1$.
    Now, let $\phi \in \mathrm{HC}^k_{\mathrm{diff}}(M, \Diffop(P)^G)$
    be closed, $\delta \phi = 0$. There exists an $l \in \mathbb{N}_0$
    and a multi-index $L\in \mathbb{N}^k_0$ with $\phi \in
    \Diffop^L(M, \Diffop^l(P)^G)$. The restrictions $\phi_{U_\alpha}
    \in \Diffop^L (U_\alpha, \Diffop^l (\mathsf{p}^{-1}
    (U_\alpha))^G)$ are closed, too. For them we find $\Theta_\alpha
    \in \Diffop^{\tilde{L}} (U_\alpha, \Diffop^{l+1} (\mathsf{p}^{-1}
    (U_\alpha))^G)$, $\tilde{L}\in\mathbb{N}_0^{k-1}$ as in
    \eqref{eq:new-multi-index}, with $\delta \Theta_\alpha=
    \phi_{U_\alpha}$. The restrictions
    \[
    \Theta_{\chi_\alpha}(a_1, \ldots, a_{k-1})|_{\mathsf{p}^{-1}(U_\alpha)}
    = 
    \chi_\alpha|_{U_\alpha}
    \cdot
    \Theta_\alpha(a_1|_{U_\alpha}, \ldots, a_{k-1}|_{U_\alpha})
    \]
    and $\Theta_{\chi_\alpha} (a_1,\ldots,a_{k-1})|_{P\setminus
      \mathsf{p}^{-1}(\supp \chi_\alpha)} = 0$ define global elements
    $\Theta_{\chi_\alpha} \in
    \Diffop^{\tilde{L}}(M,\Diffop^{l+1}(P)^G)$ which are evidently
    $G$-invariant.  As in the proof of Theorem
    \ref{theorem:Hochschild-cohomologies-surjective-submersions} one
    shows that $\Theta = \sum_{\alpha \in I}\Theta_{\chi_\alpha} \in
    \mathrm{HC}^{k-1}_{\mathrm{diff}}(M,\Diffop^{l+1}(P))$ has the
    correct degrees of differentiation and yields
    \eqref{eq:phiExactInvariant}.
\end{proof}
\begin{corollary}
    Every principal fibre bundle admits a deformation quantization
    which is unique up to equivalence.
\end{corollary}

In order to prove the existence of a deformation quantization which in
addition preserves the fibration, we proceed analogously to the
general case, only taking care of the additional $G$-invariance. We
have to choose an always existing $G$-invariant, torsion-free
covariant derivative $\nabla^P$ respecting the vertical bundle. The
additional requirement of $G$-invariance means that $\mathsf{r}_g^*
\nabla^P_Z W = \nabla^P_{\mathsf{r}_g^*Z} \mathsf{r}_g^* W$ for all
vector fields $V, W \in \Gamma^\infty(TP)$ and $g\in G$.
\begin{lemma}
    \label{lemma:covariant-derivative-G-invariant}
    Let $\nabla^P$ be a $G$-invariant covariant
    derivative. Then the $l$-fold symmetrized covariant derivative is
    $G$-invariant for all $l\in \mathbb{N}$,
    \begin{equation}
        \label{eq:covariant-derivative-invariant}
        \mathsf{r}_g^* \circ \mathsf{D}_P^{(l)}= \mathsf{D}_P^{(l)}
        \circ \mathsf{r}_g^*. 
    \end{equation}
\end{lemma}
\begin{lemma}
    Let $\bullet$ be a deformation quantization of a principal fibre
    bundle $\mathsf{p}: P \longrightarrow M$ with structure Lie group
    $G$. Then there exists a formal series $T = \id +
    \sum_{r=1}^\infty \lambda^r T_r$ of $G$-invariant
    differential operators $T_r \in \Diffop(P)^G$, such that
    \begin{equation}
        \label{eq:Diffop-deformation-quantization-G-invariant}
        T(\mathsf{p}^*a) = 1 \bullet a 
        \quad \forall a\in C^\infty(M).
    \end{equation}
\end{lemma}
\begin{proof}
    According to the proof of Lemma \ref{lemma:special-diffop} the
    $G$-invariance of $\bullet$ yields $\mathsf{r}_g^* D_r (a) =
    D_r(a)$ and $T^r_s\in \Gamma^\infty(HP)^G$. The $G$-invariance of
    $\nabla^P$ finally proves the assertion.
\end{proof}
Again, such a map $T$ is an equivalence transformation and leads to
the following corollary which concludes the proof of Theorem
\ref{theorem:def-quant-principal}.
\begin{corollary}
    Every principal fibre bundle admits a deformation quantization
    which preserves the fibration.
\end{corollary}

Since every principal fibre bundle is a surjective submersion, every
deformation quantization $\bullet$ and every decomposition
\eqref{eq:decomposition-diffop} lead to a bimodule structure
\eqref{eq:bimodule-surjective-submersion}. As we will see, there
exists a geometrically motivated choice of the decomposition
\eqref{eq:decomposition-diffop} such that the derived structures
$\bulletp$ and $\starp$ are $G$-invariant with respect to the
canonical left action of $G$ on $\Diffop(P)$ induced by
$\mathsf{r}_g^*$.  Clearly, $\Diffopver(P)$ is an invariant subspace.
\begin{lemma}
    \label{lemma:iso-diffop-surj-sub}
    Let $\mathsf{p}: P\longrightarrow M$ be a surjective submersion
    and let $\nabla^P$ be a torsion-free covariant
    derivative, which respects the vertical bundle. Then
    the symbol map
    \begin{equation}
        \label{eq:symbol-diffops}
        \sigma: \Diffop^\bullet (P)\longrightarrow
        \Gamma^\infty(\mathrm{S}^\bullet TP)=\bigoplus_{l=0}^\infty
        \Gamma^\infty(\mathrm{S}^l TP),
    \end{equation}
    with respect to $\nabla^P$ restricts to a vector space
    isomorphism
    \begin{equation}
        \label{eq:symbol-vertical-diffops}
        \sigma: \Diffopver(P) 
        \longrightarrow
        \Gamma^\infty(S^\bullet VP) 
        = \bigoplus_{l=0}^\infty \Gamma^\infty(\mathrm{S}^l VP)
    \end{equation}
    from the vertical differential operators to the vertical symmetric
    multi-vector fields.  If in addition $\mathsf{p}:P\longrightarrow
    M$ is a principal fibre bundle and $\nabla^P$ is $G$-invariant,
    then $\sigma$ is $G$-equivariant.
\end{lemma}
\begin{proof}
    That \eqref{eq:symbol-diffops} is an isomorphism is a general
    consequence of any symbol calculus, confer
    Remark~\ref{remark:symbol-calculus}. In the present situation we
    additionally choose a connection on $P$ and consider the covariant
    derivative $\nabla^M$ of
    Lemma~\ref{lemma:connection-covariant-derivatives}. With
    \eqref{eq:symmetric-covariant-derivatives-relation} and the
    Leibniz rule $\mathsf{D}_P^{(j)}(fh) = \sum_{s=0}^j
    \mathsf{D}_P^{(s)} f \vee \mathsf{D}_P^{(j-s)} h$ we find that the
    property $D(f\cdot \mathsf{p}^*a)=D(f)\cdot \mathsf{p}^*a$ of
    vertical differential operators $D=\sum_{j=0}^l D_{T_j}\in
    \Diffopver^l (P)$ is equivalent to $T_j\in
    \Gamma^\infty(\mathrm{S}^jVP)$ for all $j=0,\ldots ,l$.  Due to
    Lemma~\ref{lemma:covariant-derivative-G-invariant}, $G$-invariance
    of $\nabla^P$ finally leads to $\mathsf{r}_g^* \circ D_{T_{j}} =
    D_{(\mathsf{r}^g)^*T_{j}} \circ \mathsf{r}_g^*$ for all $g \in G$
    and $j = 0, \ldots , l$.
\end{proof}

The choice of a principal connection, i.e.\  a $G$-invariant
decomposition $TP = VP \oplus HP$ with $T \mathsf{r}_g HP= HP$, gives
rise to a $G$-invariant decomposition
\begin{equation}
    \label{eq:decomposition-vector-fields}
    \Gamma^\infty(\mathrm{S}^\bullet TP) =
    \Gamma^\infty(\mathrm{S}^\bullet VP) \oplus
    \overline{\Gamma^\infty(\mathrm{S}^\bullet VP)},
\end{equation}
where $\overline{\Gamma^\infty(\mathrm{S}^\bullet VP)}$ is the
complementary vector space of $\Gamma^\infty(\mathrm{S}^\bullet VP)$
induced by $TP = VP \oplus HP$.  Using the symbol map of
Lemma~\ref{lemma:iso-diffop-surj-sub} with respect to a $G$-invariant
covariant derivative, $\overline{\Diffopver(P)} =
\sigma^{-1}(\overline{\Gamma^\infty(\mathrm{S}^\bullet VP)})$ is a
vector space which leads to a particular decomposition
\eqref{eq:decomposition-diffop}
\begin{equation}
    \label{eq:decomposition-diffop-special}
    \Diffop(P) = 
    \Diffopver (P) \oplus
    \overline{\Diffopver(P)}.
\end{equation}
Due to $G$-equivariance of $\sigma$ this decomposition is
$G$-invariant, too.  Obviously,
\eqref{eq:decomposition-diffop-special} depends on the choices of the
principal connection and the covariant derivative.  Altogether, this
leads to the final theorem of this section.
\begin{theorem}
    \label{theorem:CooleDeformation}
    Let $\mathsf{p}: P\longrightarrow M$ be a principal fibre bundle
    with structure Lie group $G$ and let $\bullet$ be a deformation
    quantization with respect to a star product $\star$ on $M$. Then
    there exists a bimodule structure
    \begin{equation}
        \label{eq:bimodule-principal}
        {}_{( \Diffopver (P)
          [[\lambda]],\starp)} (\bulletp,C^\infty(P)
        [[\lambda]],\bullet)_{(C^\infty(M)[[\lambda]],\star)}
    \end{equation}
    as in \eqref{eq:bimodule-surjective-submersion} with the further
    property that $\starp$ and $\bulletp$ are $G$-invariant, i.e.
    \begin{equation}
        \label{eq:starp-G-invariant}
        \mathsf{r}_g^* (A \starp B)
        = (\mathsf{r}_g^* A) \starp (\mathsf{r}_g^* B),
    \end{equation}
    \begin{equation}
        \label{eq:bulletp-G-invariant}
        \mathsf{r}_g^* (A \bulletp f) 
        =  (\mathsf{r}_g^* A)\bulletp (\mathsf{r}_g^*f)
    \end{equation}
    for all $A, B \in \Diffopver(P)[[\lambda]]$, $f \in
    C^\infty(P)[[\lambda]]$ and $g \in G$. Moreover, $\starp$ is
    unique up to $G$-equivariant equivalence and the algebras in
    \eqref{eq:bimodule-principal} are mutual commutants.
\end{theorem}
\begin{proof}
    Regarding the recursive construction of $\rho'= \id +
    \sum_{r=1}^\infty \lambda^r \rho'_r$ from
    Proposition~\ref{proposition:Kommutante} an induction shows that
    $\rho'$ is $G$-invariant: indeed, the $G$-invariance of $\bullet$
    and $\mathsf{r}_g^* (\delta D (a)) = \delta(\mathsf{r}_g^* D)(a)$
    for all $D \in \Diffop(P)$ and $a \in C^\infty(M)$ show that for
    all $A \in \Diffopver(P)$ and $r \ge 1$ the element
    $\mathsf{r}_g^* \rho'_r(A)$ satisfies the same defining equation
    as $\rho'_r(\mathsf{r}_g^* A)$. With the $G$-invariance of the
    decomposition \eqref{eq:decomposition-diffop-special} these
    elements are equal by uniqueness. Thus we have the $G$-invariance
    of $\rho'$ and \eqref{eq:bulletp-G-invariant} follows directly. By
    definition, this implies \eqref{eq:starp-G-invariant}. The
    uniqueness of $\starp$ up to $G$-equivariant equivalence is also
    clear as $\bulletp$ establishes a $G$-equivariant isomorphism to
    the commutant of $(C^\infty(M)[[\lambda]], \star)$.
\end{proof}
\begin{remark}
    \label{remark:DeformedUniversalEnvelopping}
    From this point of view one can now understand the results of
    \cite{jurco.schraml.schupp.wess:2000a} in a more geometric way,
    independent of the concrete model of the noncommutative gauge
    theory: For two infinitesimal classical gauge transformations
    $\xi, \eta \in \mathfrak{gau}(P) = \Gamma^\infty(VP) \subseteq
    \Diffop_{\mathrm{ver}}^1(P)$ the deformed action gives
    \begin{equation}
        \label{eq:xietaCommutator}
        \xi \bulletp (\eta \bulletp f) - \eta \bulletp (\xi \bulletp f)
        = (\xi \starp \eta - \eta \starp \xi) \bulletp f
        = \Lie_{[\xi, \eta]} f + \cdots,
    \end{equation}
    where in general the higher order contributions are
    \emph{nontrivial}. The Lie algebra of infinitesimal gauge
    transformations $\mathfrak{gau}(P)$ is in general no longer a Lie
    subalgebra of $(\Diffopver(P)[[\lambda]], \starp)$ with respect to
    the $\starp$-commutator. This was observed in
    \cite{jurco.schraml.schupp.wess:2000a} on the level of associated
    vector bundles for particular situations.
\end{remark}

%
%

\section{Associated Vector Bundles}
\label{sec:AssociatedBundles}

The construction of associated vector bundles is one of the most
important features of principal fibre bundles. The present section
shows that our Definition~\ref{definition:def-quant-principal} of a
deformation quantization of principal fibre bundles naturally leads to
a deformation quantization of associated vector bundles.

Deformation quantization of vector bundles
$\mathsf{p}:E\longrightarrow M$ is already established and can be put
down to deformation quantization of finitely generated projective
modules, since the sections $\Gamma^\infty(E)$ are such a right module
over $C^\infty(M)$. The well-known definitions and results can all be
found in \cite{bursztyn.waldmann:2000b} and \cite{waldmann:2005a}.
Altogether, one considers the deformed bimodule
\begin{equation}
    \label{eq:deformed-bimodule-vector-bundle}
    {}_{(\Gamma^\infty(\End(E)) [[\lambda]],\star'_E)} (\bullet'_E,
    \Gamma^\infty(E) [[\lambda]], \bullet)_ {(C^\infty(M)
      [[\lambda]],\star)}, 
\end{equation}
where the right module structure itself is defined as a deformation
quantization of the vector bundle $E$.  It is known, confer
\cite{waldmann:2005a}, Thm.~1, that for a given star product $\star$
the deformed bimodule structures $\bullet'_E$, $\bullet$ and the
algebra structure $\star'_E$ are unique up to equivalence and that the
two deformed algebras $(\Gamma^\infty(\End(E)) [[\lambda]],\star'_E)$
and $(C^\infty(M) [[\lambda]],\star)$ are mutual commutants.

In the following we use some basic facts about the geometry of
associated vector bundles which for example can be found in detail in
\cite{kolar.michor.slovak:1993a}. Let $\mathsf{p}:P\longrightarrow M$
be a principal fibre bundle with structure Lie group $G$.  The
associated vector bundle with respect to a representation $\pi:
G\longrightarrow \End(V)$ of $G$ on a finite dimensional vector space
$V$ over $\mathbb{C}$ from the left is denoted by $P\times_G V$.  All
results of this section are based on the well-known isomorphism
\begin{equation}
    \label{eq:isomorphism-sections}
    \Gamma^\infty(P\times_G V)
    \cong
    \left(C^\infty(P)\otimes V\right)^G
\end{equation}
between the smooth sections of the associated bundle and the
$G$-invariant $V$-valued functions on $P$ with respect to the left
action $\mathsf{r}_g^* \otimes \pi(g)$ on $C^\infty(P) \otimes V$.
Together with the according left action on the algebra $\Diffop(P)
\otimes \End(V)$ one finds the following lemma.
\begin{lemma} 
    \label{lemma:extended-bimodule}
    The $G$-invariant bimodule structure 
    \begin{equation}
        \label{eq:bimodule-principal-again}
        {}_{( \Diffopver (P)
          [[\lambda]],\starp)} (\bulletp,C^\infty(P)
        [[\lambda]],\bullet)_{(C^\infty(M)[[\lambda]],\star)} 
    \end{equation}
    of Theorem~\ref{theorem:CooleDeformation} yields a
    bimodule structure 
    \begin{equation}
        \label{eq:extended-bimodule}
        {}_{((\Diffopver (P)\otimes \End(V))^G [[\lambda]],\starp)}
        (\bulletp, (C^\infty(P)\otimes V)^G[[\lambda]],
        \bullet )_ 
        {(C^\infty(M) [[\lambda]],\star)}.
    \end{equation}
\end{lemma}
\begin{proof}
    With the obvious new structures the $G$-invariance of $\starp$,
    $\bulletp$ and $\bullet$ ensures that the considered $G$-invariant
    elements form a sub-algebra and a sub-bimodule.
\end{proof}

Using the isomorphism \eqref{eq:isomorphism-sections}, the right module
structure of \eqref{eq:extended-bimodule} directly leads to a
deformation quantization of the associated vector bundle.
\begin{theorem}[Deformation quantization of associated vector bundles]
    Let $\bullet$ be a deformation quantization of a principal fibre
    bundle. Then, by \eqref{eq:extended-bimodule}, every associated
    vector bundle inherits a deformation quantization in the sense of
    \cite{bursztyn.waldmann:2000b}.
\end{theorem}

Since \eqref{eq:extended-bimodule} is a bimodule structure, the
algebra $((\Diffopver (P)\otimes \End(V))^G
[[\lambda]],\starp)$ can be related to the commutant of the deformed
vector bundle. To this end we first recall that $P\times_G \End(V) \cong
\End(P\times_G V)$, where $\End(P\times_G V)$ is the bundle of
endomorphisms of $P\times_G V$. One gets 
\begin{equation}
    \label{eq:isomorphism-sections-endos}
    (C^\infty(P)\otimes \End(V))^G \cong
    \Gamma^\infty(\End(P\times_G V))
\end{equation}
and with \eqref{eq:isomorphism-sections} the action of $f\otimes
L\in(C^\infty(P)\otimes \End(V))^G$ on $h\otimes w\in
(C^\infty(P)\otimes V)^G$ is given by $(f\otimes L)(h\otimes
w)= fh\otimes L(w)$. In analogy to that one defines the map
\begin{equation}
    \label{eq:map-on-endos}
    \psi: (\Diffopver(P) \otimes \End(V))^G
    \longrightarrow \Gamma^\infty(\End(P\times_G V))
\end{equation}
via $(A\otimes L)(h\otimes w)=A(h)\otimes L(w)$. By
\eqref{eq:isomorphism-sections-endos} $\psi$ is surjective. 
Altogether we have the commutative diagram
\begin{equation}
    \label{eq:diagramm-commutant}
    \xymatrix{
    (C^\infty(P)\otimes \End(V))^G \ar@{_{(}->}[d]
    \ar^{\cong}[rd]& \\
    (\Diffop_{\mathrm{ver}}(P)\otimes \End(V))^G \ar[r]^(.65){\psi} &
    **[r]\Gamma^\infty(\End(P\times_G V)).
    }
\end{equation}
With these arrangements we can investigate the commutant of the
deformed associated vector bundle.  By \eqref{eq:extended-bimodule},
for every element $D\in (\Diffopver(P) \otimes
\End(V))^G[[\lambda]]$ there exists a unique element $\phi(D)\in
\End_{C^\infty(M)[[\lambda]]}(\Gamma^\infty(P\times _G
V)[[\lambda]], \bullet)$ in the commutant with $D \bulletp s = \phi(D)s$ for
all $s\in \Gamma^\infty(P\times _G V)[[\lambda]]$.
\begin{lemma}
    The map
    \begin{equation}
        \label{eq:surjection-commutant}
        \phi: (\Diffopver(P) \otimes
        \End(V))^G[[\lambda]]
        \longrightarrow 
        \End_{C^\infty(M)[[\lambda]]}(\Gamma^\infty(P\times
        _G V) [[\lambda]], \bullet)
    \end{equation}
    is surjective.
\end{lemma}
\begin{proof}
    Let $K=\sum_{r=0}^\infty \lambda^r K_r \in
    \End_{C^\infty(M)[[\lambda]]} (\Gamma^\infty(P\times _G V)
    [[\lambda]], \bullet)$. By definition, it is clear that $K_0\in
    \Gamma^\infty(\End(P\times_G V))$. With \eqref{eq:map-on-endos}
    there exists an element $D_0\in (\Diffopver(P) \otimes \End(V))^G$
    with $K_0=\psi(D_0)$. With $\phi(D_0)= \sum_{r=0}^\infty\lambda^r
    \phi(D_0)_r$ it follows that $\psi(D_0)=\phi(D_0)_0$. Thus, the
    element $K- \phi(D_0)= \sum_{r=1}^\infty \lambda^r K_r^{(1)}$ of
    the commutant begins in order $\lambda$. Due to the
    $\mathbb{C}[[\lambda]]$-linearity of $\phi$, iteration proves the
    lemma.
\end{proof}

By the results of \cite{waldmann:2005a}, the commutant
$\End_{C^\infty(M)[[\lambda]]}(\Gamma^\infty(P\times _G V)[[\lambda]],
\bullet)$ is isomorphic to $\Gamma^\infty(\End(P\times_G
V))[[\lambda]]$ and we immediately get the following theorem.
\begin{theorem}
    \label{theorem:associated-commutant}
    Let $\bullet$ denote a deformation quantization of a principal
    fibre bundle as well as the induced deformation quantization of an
    associated vector bundle $E=P\times_G V$. Then, for all structures
    $\star'_E$ and $\bullet'_E$ as in
    \eqref{eq:deformed-bimodule-vector-bundle} there exists a
    surjective algebra homomorphism
    \begin{equation}
        \label{eq:surjection-endos}
        \phi: ((\Diffopver(P) \otimes
        \End(V))^G[[\lambda]],\starp)
        \longrightarrow 
        (\Gamma^\infty(\End(E)) [[\lambda]], \star'_E)
    \end{equation}
    such that
    \begin{equation}
        \label{eq:property-surjection-endos}
        D \bulletp s = \phi(D)\bullet'_E s
    \end{equation}
    for all $D\in (\Diffopver(P) \otimes
    \End(V))^G[[\lambda]]$ and $s\in \Gamma^\infty(E)[[\lambda]]$.
\end{theorem}

We conclude this section with a few remarks on aspects concerning
Morita theory. It is well known that classically $\Gamma^\infty(E)$
provides a Morita equivalence bimodule between the algebras
$C^\infty(M)$ and $\Gamma^\infty(\End(E))$, provided $E$ has non-zero
fibres. The corresponding deformed bimodule
\eqref{eq:deformed-bimodule-vector-bundle} is still a Morita
equivalence bimodule, see the discussions in
\cite{bursztyn.waldmann:2000b, bursztyn:2002a,
  bursztyn.waldmann:2005b}, where also a \emph{strong} version of
Morita equivalence including the $^*$-algebra aspects is discussed.

From this point of view, Theorem~\ref{theorem:associated-commutant}
says that all the Morita equivalent algebras of the form
$(\Gamma^\infty(\End(E))[[\lambda]], \star'_E)$ are obtained from
$(\Diffopver(P) \otimes \End(V))^G[[\lambda]], \star')$. Motivated by
the diagram \eqref{eq:diagramm-commutant}, one can now ask the
following question: Is $(C^\infty(P) \otimes \End(V))^G$ deformed into
a \emph{subalgebra} of $(\Diffopver(P) \otimes \End(V))^G[[\lambda]],
\star')$ such that under the map $\phi$ it becomes isomorphic to
$(\Gamma^\infty(\End(E))[[\lambda]], \star'_E)$?  Classically, this is
clearly the case, but for the deformed situation there are
obstructions: we call the principal bundle $P$ \emph{sufficiently
  non-trivial} if there exists at least one non-trivial representation
$\pi$ of $G$ on $\mathbb{C}$ such that the line bundle $L = P \times_G
\mathbb{C}$ over $M$ has non-trivial first Chern class $c_1 (L) \in
\HdR^2(M, \mathbb{C})$. Here we have to use the image of the
topologically defined Chern class in the deRham cohomology, i.e.\ 
possible torsion effects may be lost.
\begin{theorem}
    \label{theorem:MoritaObstruction}
    Assume that the principal bundle $P$ is sufficiently non-trivial
    and $\star$ is a symplectic star product. Then there exists no
    deformation $\star'_M$ of $C^\infty(M)$ with an algebra homomorphism
    \begin{equation}
        \label{eq:starPrimeMtostarPrime}
        \Phi: \left(C^\infty(M)[[\lambda]], \star'_M\right)
        \longrightarrow
        \left(\Diffopver(P)^G[[\lambda]], \star'\right)
    \end{equation}
    with $\Phi = \mathsf{p}^* + \sum_{r=1}^\infty \lambda^r \Phi_r$.
\end{theorem}
\begin{proof}
    Assume that such a star product $\star'_M$ and a corresponding
    homomorphism $\Phi$ exist. Then $(C^\infty(M)[[\lambda]],
    \star'_M)$ would act from the left on the section
    $\Gamma^\infty(E)[[\lambda]]$ for any associated bundle according
    to Theorem~\ref{theorem:associated-commutant}. Moreover, this left
    module structure deforms the classical action as all the maps
    $\Phi$ and $\phi$ in zeroth order combine to the classical left
    multiplication of sections of $E$ with functions. Applying this to
    the case of a line bundle $L$ we know from
    \cite{bursztyn.waldmann:2002a}, Thm.~1, that $\star'_M$
    necessarily quantizes the same (symplectic) Poisson bracket and
    that the \emph{relative class}, see\cite{gutt.rawnsley:1999a}, of
    $\star$ and $\star'_M$ is given by $t(\star'_M, \star) = 2 \pi \I
    c_1(L)$.  This has to be valid for \emph{all} line bundles
    obtained from association. Thus if there exists a line bundle $L$
    with non-trivial first Chern class, the star products $\star'_M$
    and $\star$ are in-equivalent.  Conversely, one always has the
    trivial representation of $G$ resulting in the trivial line bundle
    $L = M \times \mathbb{C}$, from which we conclude that the
    relative class vanishes, i.e.\ $\star'_M$ and $\star$ are
    equivalent. This is a contradiction.
\end{proof}
\begin{example}
    \label{example:HopfFibrationAgain}
    Again, the Hopf fibration $S^3 \longrightarrow S^2$ provides an
    example for a sufficiently non-trivial principal bundle. The
    $\mathrm{U}(1)$-representations $(\E^{\I\phi}, z) \mapsto
    \E^{n\I\phi} z$ for $n \in \mathbb{Z}$ are all non-isomorphic and
    yield non-isomorphic line bundles $L_n$. In fact, $L_n =
    L_1^{\otimes n}$ for $n \ge 1$ and $L_{-1} = L_1^*$. The Chern
    classes are $c_1(L_n) = n c_1(L_1) \ne 0$ for $n \ne 0$.
\end{example}

This theorem can also be seen as a refined version of the obstruction
for deforming $\mathsf{p}^*$ into an algebra homomorphism as discussed
in the introduction. In fact, Theorem~\ref{theorem:MoritaObstruction}
gives an obstruction for a \emph{bimodule} structure on
$C^\infty(P)[[\lambda]]$ with respect to possibly two different star
products $\star$ and $\star'_M$.
\begin{corollary}
    \label{corollary:NoBimodule}
    Let $\mathsf{p}: P \longrightarrow M$ be a sufficiently
    non-trivial principal bundle over a symplectic manifold $M$ and
    let $\star$ be a symplectic star product on $M$. Then there exists
    no deformation of $\mathsf{p}^*$ into a $G$-invariant bimodule
    structure with respect to $\star$.
\end{corollary}
\begin{proof}
    Indeed, such a bimodule structure would first give a $G$-invariant
    right module structure which is unique up to equivalence. Then the
    left module structure gives an algebra homomorphism into the
    deformed vertical differential operators which is $G$-equivariant.
    Thus the image is in $(\Diffopver (P)^G[[\lambda]],
    \star')$ which is not possible by
    Theorem~\ref{theorem:MoritaObstruction}.
\end{proof}


%
%

\begin{footnotesize}

\end{footnotesize}

\end{document}